\documentclass[12pt,a4paper]{article}
\usepackage{graphicx}
\usepackage[]{amsmath,amssymb,amsthm}
\usepackage{epstopdf}
\usepackage{amstext,amsfonts,graphicx,amscd,epsfig}

\usepackage{subfig}
\usepackage{geometry}
\newtheorem{theorem}{Theorem}[section]
\newtheorem{proposition}[theorem]{Proposition}
\newtheorem{lemma}[theorem]{Lemma}
\newtheorem{corollary}[theorem]{Corollary}
\theoremstyle{definition}
\newtheorem{definition}[theorem]{Definition}
\theoremstyle{remark}
\newtheorem{remark}[theorem]{Remark}
\newcommand{\R}{\mathbb{R}}

\newcommand{\C}{\mathbb{C}}
\newcommand{\Z}{\mathbb{Z}}

\newcommand{\inn}{\textnormal{in}}

\newcommand{\ind}{\textnormal{index}}

\newcommand{\B}{{\cal B}}
\newcommand{\M}{\text{\sf \emph{M}}}

\usepackage[usenames, dvipsnames]{color}
\usepackage{tikz}

\begin{document}

\begin{center}
{\large{\bf Stability of cycles in a game of Rock-Scissors-Paper-Lizard-Spock}}\\
\mbox{} \\
\begin{tabular}{cc}
{\bf Sofia B.\ S.\ D.\ Castro$^{\dagger,\ddagger*}$} & {\bf Ana Ferreira$^{\ddagger}$} \\
{\small sdcastro@fep.up.pt} & {\small up200800262@edu.fep.up.pt } \\
{\bf Liliana Garrido-da-Silva$^{\dagger, \ddagger}$} & {\bf Isabel S.\ Labouriau$^{\ddagger}$} \\
{\small lilianagarridosilva@sapo.pt} & {\small islabour@fc.up.pt} 
\end{tabular}

\end{center}

\noindent $^{*}$ Corresponding author.

\noindent $^{\dagger}$ Faculdade de Economia, Rua Dr.\ Roberto Frias, 4200-464 Porto, Portugal.

\noindent $^{\ddagger}$ Centro de Matem\'atica, Universidade do Porto, Rua do Campo Alegre 687, 4169-007 Porto, Portugal .

\begin{abstract}
We study a system of ordinary differential equations in $\R^5$ that is used as a model both in population dynamics and in game theory, and 
is known to exhibit a heteroclinic network 
consisting in the union of
four types of elementary heteroclinic cycles.
We show the asymptotic stability of the network for parameter values in a range compatible with 
both population and game dynamics.
We obtain estimates of the relative attractiveness of each one of the cycles by computing their stability indices.
For the parameter values ensuring the asymptotic stability of the network 
we relate the attractiveness properties of each cycle to the others. In particular, for three of the cycles we show that if one of them has a weak form of attractiveness, then the other two are completely unstable.
We also show the existence of an open region in parameter space where all four cycles are completely unstable and the network is asymptotically stable, giving rise to intricate dynamics that has been observed numerically by other authors.
\end{abstract}


\vspace{.3cm}


\section{Introduction}

The Rock-Scissors-Paper-Lizard-Spock (RSPLS, henceforth) game is an extension of the traditional Rock-Scissors-Paper  (RSP) game and has become ubiquitous\footnote{It also appears in less scientific environments such as the television show ``The Big Bang Theory''.} in the dynamical systems literature, associated especially to population dynamics. Additionally to Rock beating Scissors, Scissors beating Paper and Paper beating Rock, two more actions, Lizard and Spock, are added to construct the following relations
\begin{center}
\begin{tabular}{ccc}
\hline
Rock & wins over & Scissors \\
 & and & Lizard \\
 \hline
 Scissors & win over & Paper \\
 & and & Lizard \\
 \hline
 Paper & wins over & Rock \\
  & and & Spock \\
  \hline
 Lizard & wins over & Paper \\
  & and & Spock \\
  \hline
Spock & wins over & Rock \\
 & and & Scissors \\
 \hline
\end{tabular}
\end{center}

In the context of game theory, these are considered actions chosen by a player, while in that of population dynamics these represent types or species in a population.
In this way, each type/action wins over two other types/actions while it loses when confronted with the remaining two types/actions. These interactions can be described by the graph in Figure~\ref{fig:game}, where each node corresponds to a type or choice of an action and a directed edge indicates that the starting node beats the end node. There are dynamical systems represented by ODEs that support the dynamics of the RSPLS game such as Lotka-Volterra systems or constructed by the methods of either Field \cite{Field2015} or Ashwin and Postlethwaite \cite{AshPos2013}. The key feature is the existence of a heteroclinic network. A heteroclinic cycle is  a union of a finite number of equilibria for the ODE with the trajectories connecting them in a cyclic fashion. A network is a connected union of finitely many cycles. The equilibria correspond to the nodes and the connecting trajectories correspond to edges in the graph. 

\begin{figure}[h!]
\centering
{\includegraphics{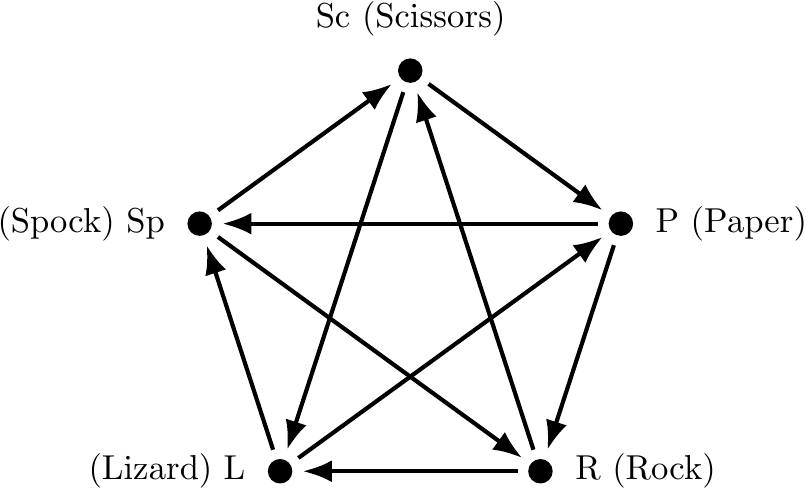}}
\caption{The RSPLS game: a directed edge indicates that the starting node beats the end node. In a heteroclinic network the connections have the opposite orientation.\label{fig:game}}
\end{figure}

A spatial version of both RSPLS and RSP is of interest to understand how different species occupy a planar finite lattice by interacting via reproduction and predatory behaviour. A given species can occupy a space in the planar lattice by either reproducing into an empty space or by predating another species occupying it. The dynamics in this instance are described by ``mean-field'' equations borrowed from physics. See the review by Szolnoki {\em et al.} \cite{SzoOliBaz} and He {\em et al.} \cite{HeTauZia}, Mowlaei {\em et al.} \cite{MowRomPlei} or Laird and Schamp \cite{LaiSch} for a description of how to derive the mean-field equations.\footnote{There is an abundance of references in the literature. We choose to mention only a few for clarity and the choice is uniquely based on our personal preferences. The reader interested in further detail and/or more examples can use the references within those we mention.}
The interactions produce a graph as above.
Parameters such as the invasion or mobility rates and reproduction rates can condition the outcome of the distribution on the lattice. 
In the language of dynamical systems these rates affect the eigenvalues of the Jacobian matrix at each node.
Important issues are those of coexistence of all available species or extinction of some species. See Park and Jang \cite{ParJan} and Kang {\em et al.} \cite{Kan_et_al2013,Kan_et_al2016a} for studies of coexistence of 5 species in a spatial version of RSPLS. Choices for the invasion and reproduction rates leading to the coexistence of some but not all the original 5 species appear in the work of Vukov {\em et al.} \cite{VukSzoSza} who extend the work of \cite{Kan_et_al2013} to contemplate more invasion rates and find that two species become extinct while the remaining three coexist. An analogous outcome is found by Cheng {\em et al.} \cite{Che_et_al2014} by looking at mesoscopic (i.e., intermediate scale) interactions whereas the modelling through PDEs supports an outcome of only two surviving species in Park {\em et al.} \cite{Par_et_al2017}.
All the results are obtained numerically.

Knebel and co-authors \cite{Kne_et_al2013,GeiKneFre} use the topology of the graph describing the interactions in each game to examine the ``interplay between the network structure and the strengths of interaction links on global stability'' and to classify coexistence networks, that is, those where all actions coexist for all strengths of the interactions. Again, a 3-action cycle seems to persist corresponding to the RSP cycle within the RSPLS game.

\begin{figure}[h!]
\centering
\subfloat[]{\includegraphics{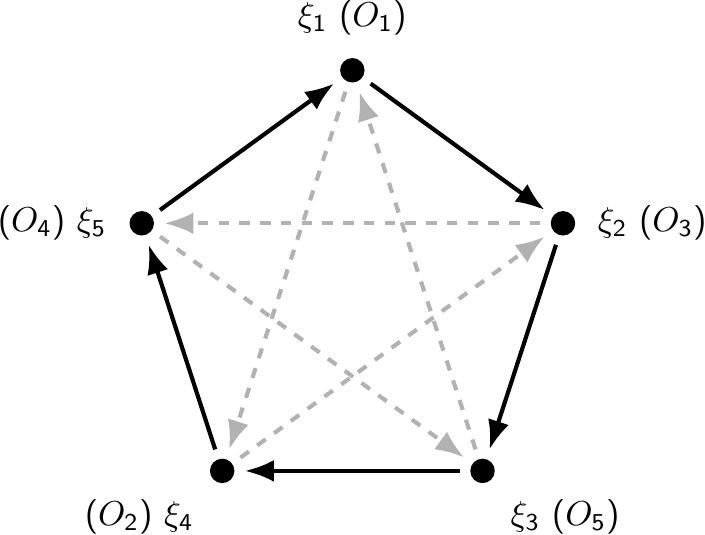}}\hfill
 \subfloat[]{\includegraphics{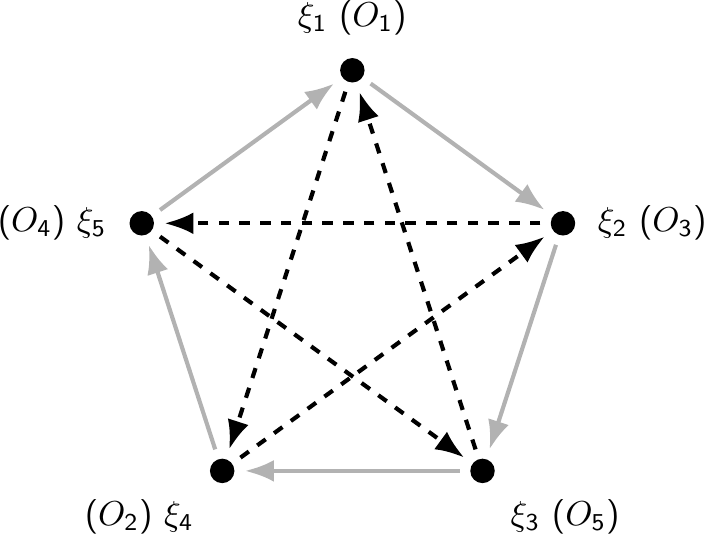}}  \\
  \subfloat[]{\includegraphics{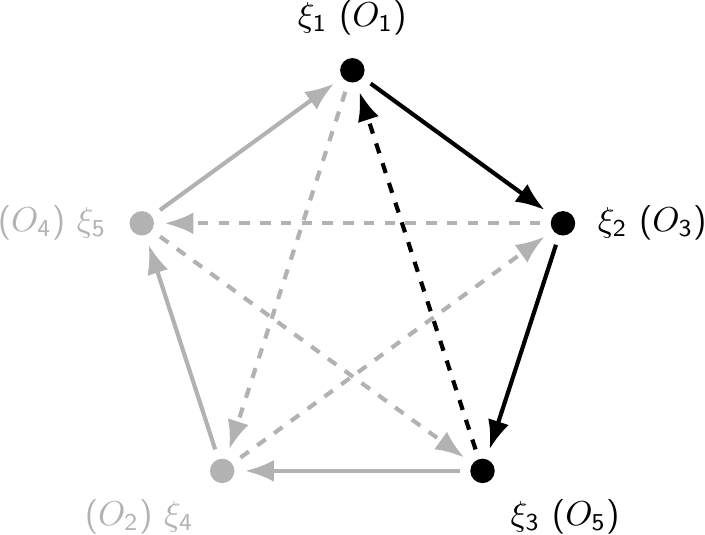}}\hfill
  \subfloat[]{\includegraphics{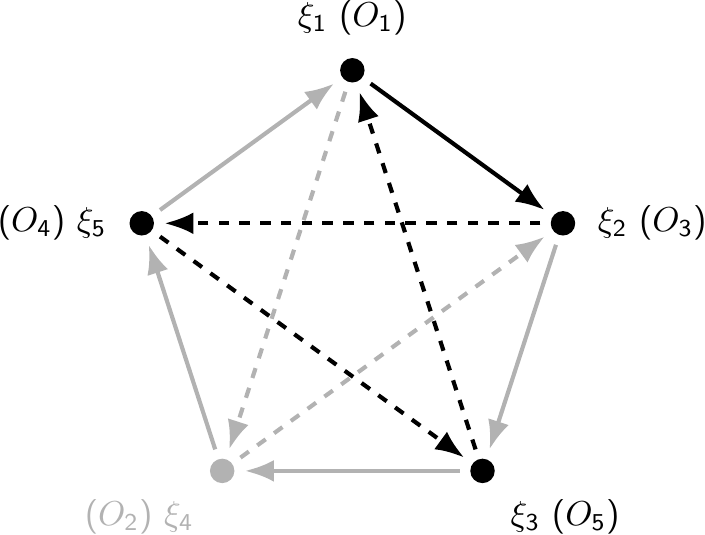}}
\caption{
(a)~The Rock-to-Paper sub-cycle; (b)~The Star or Rock-to-Spock cycle; (c)~The RSP sub-cycle; (d)~The Four-node sub-cycle.
Solid lines represent 2-dimensional connections, dashed lines are 1-dimensional. The sub-cycles are obtained by selecting one particular (1-dimensional) connecting trajectory from the 2-dimensional connection.
\label{fig:cycles}}
\end{figure}

We contribute to a theoretical understanding of the dynamics generated by the RSPLS game by studying the stability properties of 
four distinguished cycles in this network. Namely, see Figure~\ref{fig:cycles},
\begin{itemize}
\item the Rock-to-Paper cycle corresponding to the cyclic dominance of Rock over Lizard, Lizard over Spock, Spock over Scissors, Scissors over Paper, and finally Paper over Rock;
\item the Rock-to-Spock or Star cycle corresponding to the other cyclic dominance among the actions, namely, Rock over Scissors, Scissors over Lizard, Lizard over Paper, Paper over Spock, and at last Spock over Rock;
\item the RSP cycle corresponding in the above literature to the coexistence of only three of the five species;
\item the Four-node cycle corresponding to the coexistence of four of the five species.
\end{itemize}

The Rock-to-Paper sub-cycle consists of the trajectories (1-dimensional) in the 2-dimensional connections (solid lines in Figure~\ref{fig:cycles}, referred to as of type A in \cite{PosRuc}) that are contained in coordinate planes. In the Star or Rock-to-Spock cycle all connections are 1-dimensional (dashed lines in Figure~\ref{fig:cycles},  referred to as of type B in \cite{PosRuc}). 
The RSP sub-cycle has two trajectories which are part of two connections of dimension~2 and one connection of dimension~1. It corresponds to a sequence AAB in \cite{PosRuc}.
The Four-node sub-cycle comprises one trajectory belonging to one connection of dimension~2 and three connections of dimension~1. It corresponds to a sequence $\text{Q}=\text{ABBB}$ in \cite{PosRuc}.
The RSP and Four-node cycles appear each in five equivalent configurations as follows:
\begin{align*}
& \mbox{\bf Rock} \rightarrow \mbox{\bf Scissors} \rightarrow \mbox{\bf Paper} \rightarrow \mbox{Rock} \\
& \mbox{Paper} \rightarrow \mbox{Rock} \rightarrow \mbox{Lizard} \rightarrow \mbox{Paper} \\
& \mbox{Scissors} \rightarrow \mbox{Paper} \rightarrow \mbox{Spock} \rightarrow \mbox{Scissors} \\
&\mbox{Spock} \rightarrow \mbox{Scissors} \rightarrow \mbox{Lizard} \rightarrow \mbox{Spock} \\
& \mbox{Lizard} \rightarrow \mbox{Spock} \rightarrow \mbox{Rock} \rightarrow \mbox{Lizard}
\intertext{and}
& \mbox{\bf Rock} \rightarrow \mbox{\bf Scissors} \rightarrow \mbox{\bf Paper} \rightarrow \mbox{\bf Lizard} \rightarrow \mbox{Rock} \\
& \mbox{Paper} \rightarrow \mbox{Spock} \rightarrow \mbox{Scissors} \rightarrow \mbox{Lizard} \rightarrow \mbox{Paper} \\
&\mbox{Scissors} \rightarrow \mbox{Lizard} \rightarrow \mbox{Spock} \rightarrow \mbox{Rock} \rightarrow \mbox{Scissors} \\
& \mbox{Spock} \rightarrow \mbox{Rock} \rightarrow \mbox{Lizard} \rightarrow \mbox{Paper} \rightarrow \mbox{Spock} \\
& \mbox{Lizard} \rightarrow \mbox{Paper} \rightarrow \mbox{Rock} \rightarrow \mbox{Scissors} \rightarrow \mbox{Lizard}.
\end{align*}

Although heteroclinic cycles in a network cannot be asymptotically stable, they may exhibit weaker notions of stability such as {\em fragmentary asymptotic stability} (f.a.s.) and {\em essential asymptotic stability} (e.a.s.). See Podvigina \cite{Podvigina2012} and Melbourne \cite{Melbourne1991}, respectively. 
The notion of e.a.s.\ is strong enough to allow e.a.s.\ 
cycles to be visible in simulations. An f.a.s., but not e.a.s., cycle is frequently (but not always) too weak to be spotted in simulations or experiments. However, if the whole network is asymptotically stable it attracts all nearby trajectories. Less stable cycles in an asymptotically stable network may thus become visible.
We put together previously established  and new results concerning stability of networks and cycles (see Podvigina {\em et al.} \cite{PodCasLab2020} and Garrido-da-Silva and Castro \cite{GarCas2019}) to study the stability of the entire network and of the 
four heteroclinic (sub-)cycles listed above. Our results provide a theoretical background for some of the numerical observations in the literature.

By resorting to the representation using Lotka-Volterra systems available in Afraimovich {\em et al.}~\cite{AfrMosYou} our study of the stability of the cycles in the network contributes also to a deeper understanding of the notable results obtained by Postlethwaite and Rucklidge \cite{PosRuc}. We note that, for the parameter values used in \cite{PosRuc}, neither the stability conditions of \cite{AfrMosYou} nor those of \cite{PodCasLab2020} provide a positive result.

We establish a weaker condition than that obtained in \cite{AfrMosYou} that is nevertheless sufficient to ensure the asymptotic stability of the RSPLS network as a whole. This supports the visibility of cycles which are only weakly stable (f.a.s.) in \cite{PosRuc}. We provide a thorough study of the stability of the 
four (sub-)cycles, 
Rock-to-Paper, Star, RSP and Four-node,
in the network as well as conditions for the interested reader to assert the stability of any other cycle. We note that our results extend to models other than the Lotka-Volterra that preserve the invariance of coordinate lines and hyperplanes.

The next section gives a comprehensive overview of the relevant background and establishes the notation. Section~\ref{sec:network} provides a description of the network and clarifies the equivalence between the vector fields used in references  \cite{AfrMosYou} and \cite{PosRuc}.
Sections~\ref{sec:stability} and \ref{sec:stab-cycles} are devoted to the study of stability, the former of the network and the latter of some cycles.
Most calculations are deferred to an appendix.
The last section concludes.

\section{Background and notation}\label{sec:background}

We are interested in a dynamical system described by an ODE
\begin{equation}\label{eq:ODE}
\dot{x} = f(x),
\end{equation}
where $x \in \R^n$ and $f$ is a smooth map from $\R^n$ to itself. 
If there exists a group $\Gamma$ such that 
$$
f(\gamma . x)=\gamma . f(x) \quad \forall x\in \R^n, \;\; \gamma \in \Gamma,
$$
we say that the dynamical system \eqref{eq:ODE} is {\em $\Gamma$-equivariant}.

For each hyperbolic equilibrium $\xi$ of \eqref{eq:ODE} we denote  its stable and unstable manifolds respectively by $W^s(\xi)$ and $W^u(\xi)$.
Following Ashwin {\em et al.} \cite{AshCasLoh}, given two hyperbolic equilibria of \eqref{eq:ODE}, $\xi_i$ and $\xi_j$, we call 
$$
C_{ij} = W^u(\xi_i) \cap W^s(\xi_j),
$$
a {\em connection} from $\xi_i$ to $\xi_j$. We assume that $\xi_i$ and $\xi_j$ are neither the same equilibrium nor symmetry related so that the connection is {\em heteroclinic}. Note that if $\text{dim}(C_{ij})>1$ the connection $C_{ij}$ consists of infinitely many {\em connecting trajectories} $\kappa_{ij}= [\xi_i \rightarrow \xi_j]$, solutions of \eqref{eq:ODE} that converge to $\xi_i$ in backward time and to $\xi_j$ in forward time.

We are concerned with {\em heteroclinic cycles}, that is, with sets which are a finite union of hyperbolic saddles, $\xi_1, \hdots , \xi_m$ such that there exist connections $C_{j,j+1}$ for $j=1, \hdots , m$ with $\xi_{m+1}=\xi_1$. Generically, a connection between two saddles is not robust but when they are contained in flow-invariant spaces where the connection is of saddle-sink type, robustness is the norm. Such flow-invariant spaces appear naturally in equivariant dynamics, in the form of fixed-point spaces, as well as in game theory dynamics, in the form of either coordinate hyperplanes (Lotka-Volterra systems) or hyperfaces of a simplex (replicator dynamics). A connected union of finitely many heteroclinic cycles is a {\em heteroclinic network}.

We focus on the stability of heteroclinic cycles that are part of the same heteroclinic network. It is clear that in a network 
such that the equilibria lie on different axes 
at least one equilibrium has an unstable manifold of dimension at least 2, allowing for connections $C_{ij}$ of dimension at least 2. In such a case, the connection $C_{ij}$ often belongs to a flow-invariant space, $S$, of dimension at least 3 with two connecting trajectories, $\kappa_1$ and $\kappa_2$, in two subspaces $P_1, P_2 \subset S$ of lower dimension. We define two heteroclinic sub-cycles by distinguishing between these two connecting trajectories. Of course, when several connections are of dimension higher than 1, the combination of connecting trajectories into distinct sub-cycles increases in possibilities.

A large class of heteroclinic networks is that of {\em quasi-simple} networks whose stability properties are systematically studied by Garrido-da-Silva and Castro \cite{GarCas2019}. The stability results in \cite{GarCas2019} can be used for any heteroclinic cycle along which the return map has a particular form. Let $P_j$ be a flow-invariant 
sub-space and $\hat{L}_j$ be the vector sub-space 
spanned by $\xi_j$ in $\mathbb{R}^n$:
\begin{definition}\label{def:quasi}
A {\em quasi-simple} cycle is a robust heteroclinic cycle connecting $m<\infty$ equilibria $\xi_j \in P_j \cap P_{j-1}$ so that for all $j =1, \hdots , m$:
\begin{itemize}
 \item[(i)] $P_j$ is a flow-invariant space,
 \item[(ii)] $\text{dim}(P_j)=\text{dim}(P_{j+1})$,
 \item[(iii)] $\text{dim}(P_j \ominus \hat{L}_j)=1$, where $P_j \ominus \hat{L}_j$ is the orthogonal complement to $\hat{L}_j$ in $P_j$.
\end{itemize}
\end{definition}

Not all cycles in the RSLPS network are quasi-simple. In fact, the only quasi-simple cycle is the Star cycle. The Rock-to-Paper cycle does not satisfy (iii) and the remaining cycles do not satisfy (ii) in Definition~\ref{def:quasi}. In Section~\ref{sec:stab-cycles}, we focus on the quasi-simple 
(sub-)cycles whose connections are contained in the flow-invariant coordinate planes.

In a heteroclinic network, the strongest notion of stability we can find
is the one introduced by Melbourne \cite{Melbourne1991} \textit{essential asymptotic stability} (e.a.s.). An e.a.s. object attracts almost all trajectories that start nearby. A weaker notion of attractiveness, referred by Podvigina \cite{Podvigina2012}, is \textit{fragmentary asymptotic stability} (f.a.s.). A f.a.s. object attracts a 
positive measure set nearby, that may be very small. If a heteroclinic cycle is not, at least, 
f.a.s., then it is \textit{completely unstable} (c.u.)
and attracts almost nothing.

To make these concepts rigorous we need some notation. Let $X$ be
a compact set in $\mathbb{R}^n$ invariant under the flow $\Phi_t(x)$ of (\ref{eq:ODE}). Given a metric $d$ on $\mathbb{R}^n$ and $\epsilon>0$, an $\epsilon$-neighbourhood of $X$ is:
\[
B_{\epsilon}(X)=\left\{x\in\mathbb{R}^n: d(x,X)<\epsilon\right\}.
\]
The $\delta$-local basin of attraction of $X$ is:
\[
\mathcal{B}_{\delta}(X)=\left\{x\in\mathbb{R}^n : d(\Phi_t(x),X)<\delta \text{ for any } t\geq 0 \text{ and } \lim_{t \to \infty} d(\Phi_t(x),X)=0\right\}.
\]

\begin{definition}
The compact invariant set $X \subset \R^n$ is:
\begin{itemize}
\item  {\em essentially asymptotically stable} if 
the measure of its $\delta$-local basin of attraction, $\mathcal{B}_{\delta}(X)$, tends to full measure in a 
$\epsilon$-neighbourhood, $B_{\epsilon}(X)$, of $X$ as $\delta$ 
and $\epsilon$ become
small, that is, if
$\lim_{\delta \rightarrow 0} \left[ \lim_{\epsilon \rightarrow 0} \frac{\ell \left(B_{\epsilon}(X)  \cap \mathcal{B}_{\delta}(X) \right) }{\ell \left(B_{\epsilon}(X)\right)} \right] = 1$;

\item {\em fragmentarily asymptotically stable} if 
the measure of its $\delta$-local basin of attraction is positive, that is, if
$\ell\left(\mathcal{B}_{\delta}(X)\right)>0$ for any $\delta>0$;

\item {\em completely unstable} if there exists some $\delta>0$ such that 
the $\delta$-local basin of attraction of $X$ is of measure zero, that is, 
$\ell\left(\mathcal{B}_{\delta}(X)\right)=0$;
 \end{itemize}
 where $\ell(.)$ is the Lebesgue measure on $\R^n$. 
\end{definition}

The notion of local stability index was introduced by Podvigina and Ashwin \cite{PodviginaAshwin2011} to quantify the local extent of basins of attraction.\footnote{We ignore the subscript ``loc'' used in \cite{PodviginaAshwin2011} to distinguish between ``stability index'' and ``local stability index'' since we do not use the former.} 
Given $x \in X$, small $\delta >0$ and $\epsilon>0$, define
the relative size of the $\delta$-local basin of attraction in an $\epsilon$-neighbourhood of 
$x$ as
\[
    \Sigma_{\epsilon,\delta}(x)=\frac{\ell(B_{\epsilon}(x)\cap \mathcal{B}_{\delta}(X))}{\ell(B_{\epsilon}(x))}.
\]

\begin{definition}
For a point $x \in X$ the {\em local stability index} of $X$ at $x$ is
 \[
\sigma(x)=\sigma_+(x)-\sigma_-(x)
\]
\begin{equation*}
\text{where } \quad \sigma_+(x)=\lim_{\delta \to 0}\lim_{\epsilon \to 0} \Bigg[\frac{\ln(1-\Sigma_{\epsilon,\delta}(x))}{\ln(\epsilon)}\Bigg] \quad \text{ and } \quad \sigma_-(x)=\lim_{\delta \to 0}\lim_{\epsilon \to 0} \Bigg[\frac{\ln(\Sigma_{\epsilon,\delta}(x))}{\ln(\epsilon)}\Bigg] \qquad
\end{equation*}
\end{definition}
We use the convention that $\sigma_-(x)=\infty$ when
 $\Sigma_{\epsilon,\delta}=0$ for some $\epsilon>0$, $\delta>0$. Analogously, $\sigma_+(x)=\infty$ if there is an $\epsilon>0$ such that $\Sigma_{\epsilon,\delta}=1$. Note that $\sigma_{\pm}(x) \geq 0$, so we can assume that $\sigma(x) \in [-\infty,\infty]$; 
 the strongest form of local stability corresponds to
 $\sigma(x)=\infty$ 
 while $\sigma(x)=-\infty$ is the weakest.

A positive stability index indicates that $X$ attracts all points in the thick side of a cusp in its neighbourhood. If the stability index is negative, only points in the thin side of the cusp are attracted to $X$. See Figure~\ref{fig:stab_index}.
\begin{figure}[!hht]
\begin{center}
\parbox{60mm}{
\begin{center}
\includegraphics[width=50mm]{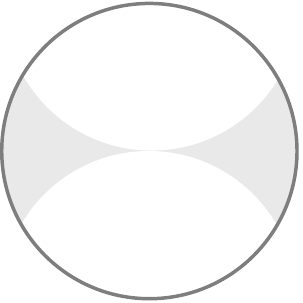}\\
 $\sigma(x)<0$
\end{center}
}
\quad
\parbox{60mm}{
\begin{center}
\includegraphics[width=50mm]{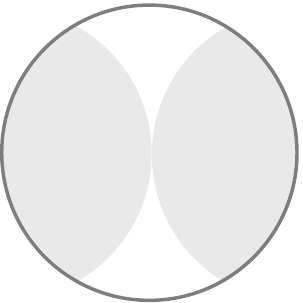}\\
$\sigma(x)>0$
\end{center}
} \quad \tikz{\fill[gray!25] (0,0) rectangle (0.8,0.4);
\node[right] at (0.9,0.2) {\small $\mathcal{B}_\delta(X)$};}
\end{center}
\caption{A negative stability index (left) indicates that the set of points in $\mathcal{\B}_\delta(X)$
are those in the thin (shaded) side of a cusp. A positive stability index (right) corresponds to 
$\mathcal{B}_\delta(X)$ being in the thick (shaded) side of a cusp.
}\label{fig:stab_index}
\end{figure}

\section{The RSPLS network}\label{sec:network}

Using the notation of \cite{AfrMosYou}, a dynamical system describing a Lotka-Volterra system is one where the ODE in \eqref{eq:ODE} takes the form\footnote{In  \cite{AfrMosYou} 
the notation is  $\sigma_i$ instead of $\tau_i$. We make this change to avoid confusion with the stability indices.} (see Equation (1) in \cite{AfrMosYou})
\begin{equation}\label{eq:RSPLS}
\dot{x}_i = x_i \left(\tau_i - \sum_{j=1}^n\; \rho_{ij}x_j \right)  \;\; \mbox{ for  } i=1,\hdots, n.
\end{equation}
All the parameters $\tau_i$ and $\rho_{ij}$ are positive and $\rho_{ii}=1$. 
To ensure biological meaning, the state space is $\R^n_+$, the subspace of $\R^n$ where all coordinates are non-negative.

For RSPLS, it is $n=5$.
In \cite{PosRuc} the dynamics of the game of RSPLS is
described by looking at a particular case of \eqref{eq:RSPLS}, namely, $\tau_j=1$ for all $j$ and 
\begin{equation}\label{eq:rho}
\rho_{j,j+1} = 1+c_A ,\
\rho_{j,j+2} = 1-e_B, \
\rho_{j,j+3} = 1+c_B, \
\rho_{j,j+4}= 1-e_A, \; (\mbox{mod }5).
\end{equation}
The dynamics of \eqref{eq:RSPLS} supports
a heteroclinic network with connections of dimension 1 and 2 between saddles. All the saddles are located on the coordinate axes and have 2-dimensional
 unstable manifolds.
  We use $O_j$ to denote equilibria when referring to the more general dynamics of \eqref{eq:RSPLS} and $\xi_j$ otherwise.
Each equilibrium $O_j$ is located at a point where only the $j^\text{th}$ coordinate is non-zero and equal to $\tau_j$. In the context of the RSPLS game, it is natural to set $\tau_j=1$ since this equilibrium represents the availability of only type $j$. 

The Jacobian matrix of the system~\eqref{eq:RSPLS} evaluated at each equilibrium $O_k$ is upper triangular with eigenvalues given by (see \cite{AfrMosYou}, Section 2)
$$
-\tau_k \;\; \mbox{ and  } \; \tau_j-\rho_{jk}\tau_k  , \;\; j \neq k.
$$
The first eigenvalue is radial and negative. In order to guarantee existence of the heteroclinic network the following assumptions are made (see Equations (3) and (4) in \cite{AfrMosYou})
\begin{equation}\label{eq:3}
\min_{i=1,2} \{ \tau_{k+i}-\rho_{k+i,k}\tau_k \} > 0
\end{equation}
and 
\begin{equation}\label{eq:4}
\tau_j-\rho_{jk}\tau_k < 0, \;\; \mbox{  for    } \; j\neq k,k+1,k+2,
\end{equation}
where all indices are~$\pmod n$. We note that in \cite{PosRuc} the first assumption holds for $i=1,3$ so that an equilibrium $\xi_k$ has connections to $\xi_{k+1}$ and $\xi_{k+3}$. The heteroclinic networks are equivalent under the following correspondence: $O_1 \equiv \xi_1$, $O_2 \equiv \xi_4$, $O_3 \equiv \xi_2$, $O_4 \equiv \xi_5$ and $O_5 \equiv \xi_3$. 

The RSPLS network is represented by the graphs
depicted in Figure~\ref{fig:network}. This is equivalent to Figure~1 in \cite{AfrMosYou} and \cite{PosRuc}, and appears in Figure~13 of \cite{PodCasLab2020}. 
Each node of the graph corresponds to an equilibrium of \eqref{eq:RSPLS} where only one type is present.  
On the right-hand side, the nodes $\xi_1, \hdots , \xi_5$ are ordered so that they correspond to the sequence Rock, Scissors, Paper, Lizard, Spock. 
On the left, they have the ordering used in  \cite{AfrMosYou}.

On the right-hand side of Figure~\ref{fig:network}, the sequence of connections $C_{j,j+1}$, (mod 5) $j=1,\hdots ,5$, together with the ordered equilibria 
constitute a heteroclinic cycle
with 2-dimensional connections. On the left, this same heteroclinic cycle is made of the sequence of connections
$C_{j,j+2}$, (mod 5) $j=1,\hdots ,5$, and the sequence of nodes $O_j$, $O_{j+2}$. We refer to this as the {\em Rock-to-Paper cycle}.

Another heteroclinic cycle consists of all the nodes (in suitable order) and the sequence of 1-dimensional connections $C_{j,j+3}$, (mod 5) $j=1,\hdots ,5$ on the right-hand side; $C_{j,j+1}$, (mod 5) $j=1,\hdots ,5$ on the left. We call this Rock-to-Spock or the {\em Star cycle}  due to its shape in the graph 
of Figure~\ref{fig:network} (right).

We point out that the Rock-Scissors-Paper game appears as the heteroclinic cycle with three consecutive nodes and the connections $C_{j,j+1}$, $C_{j+1,j+2}$, $C_{j+2,j}$
on the right-hand side; this cycle has nodes $O_j$, $O_{j+2}$, $O_{j+4}$ and the connections between each two on the left. This is the {\em RSP cycle} in what follows.

Finally, heteroclinic cycles with four nodes exist. They are described by sequences of nodes $O_j$, $O_{j+2}$, $O_{j+3}$, $O_{j+4}$ on the right-hand side; and by sequences of nodes $\xi_j$, $\xi_{j+1}$, $\xi_{j+4}$, $\xi_{j+2}$ on the left. 

 \begin{figure}[!hht]
 \centerline{\includegraphics[width=0.8\textwidth]{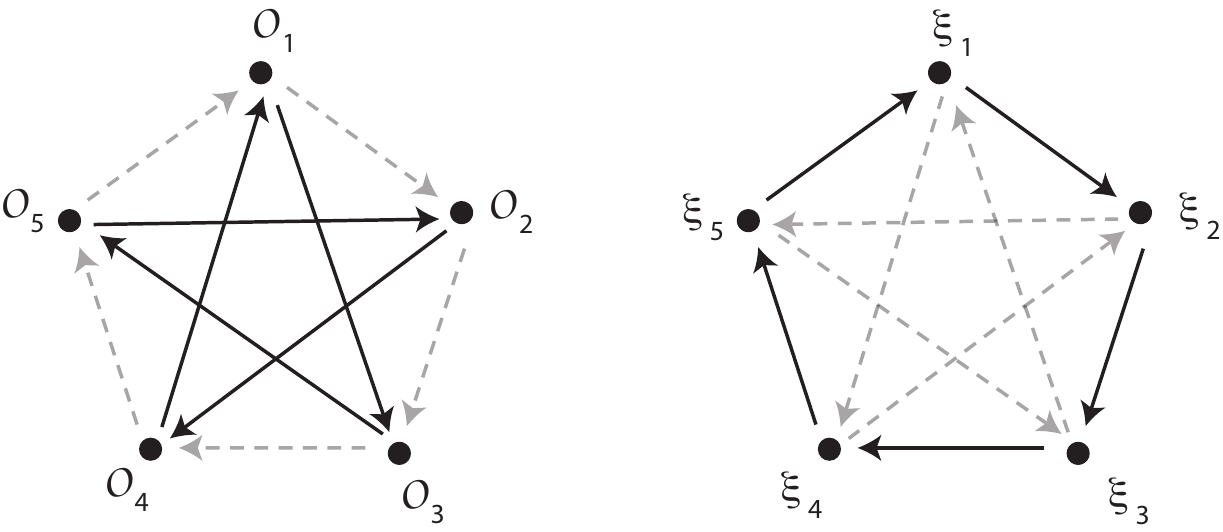}}
 \caption{\small{The RSPLS network: on the left with the labelling of \cite{AfrMosYou} and on the right with that of \cite{PosRuc}. On the left, the 2-dimensional connections are those shown as a star in the innermost part of the graph (solid lines); the connections on the outermost part, sequentially connecting $O_1, \hdots , O_5$ are all 1-dimensional (dashed lines). On the right, it is the connections on the outside of the graph (solid), connecting in sequence $\xi_1, \hdots , \xi_5$, that are 2-dimensional.} \label{fig:network}}
 \end{figure}
 
We refer to the 
four cycles described above as the {\em elementary heteroclinic cycles}.
Many heteroclinic cycles are available as combinations of these 
four types if we allow for repetition of one or more nodes. 
For instance, we may have the sequence $\xi_1\to\xi_2\to\xi_3\to\xi_4\to\xi_5\to\xi_1\to\xi_2\to\xi_3\to\xi_1$, or the sequence $\xi_1\to\xi_2\to\xi_3\to\xi_4\to\xi_5\to\xi_3\to\xi_1$, among many other.
 
 The connections among $\xi_j$, $\xi_{j+1}$ and $\xi_{j+3}$ on the right-hand side of Figure~\ref{fig:network}, namely, $C_{j,j+1}$, $C_{j+3,j+1}$, and $C_{j,j+3}$ form what Ashwin {\em et al.} \cite{AshCasLoh} call a $\Delta$-{\em clique}, as shown in Figure~\ref{fig:deltaClique}. In \cite[Definition 2.1]{PodCasLab2020}, the term $\Delta$-clique is reserved for such pieces of graph so that all trajectories starting near $\xi_j$ end at $\xi_{j+1}$. The connection $C_{j,j+1}$ is called the short-connection while $C_{j+3,j+1}$, and $C_{j,j+3}$ are the second-long and the first-long connections, respectively.
 The short connection is 2-dimensional.
 
On the left-hand side of Figure~\ref{fig:network}, the $\Delta$-cliques appear associated to the connections $C_{j,j+1}$, $C_{j+1,j+2}$, and $C_{j,j+2}$, this last being the short-connection,
 see also Figure~\ref{fig:deltaClique}.

\begin{figure}[hhh]
\centerline{\includegraphics[width=0.4\textwidth]{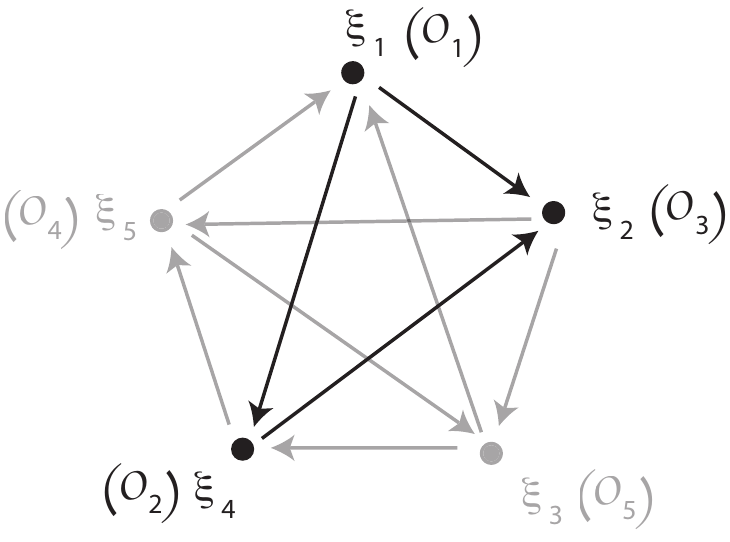}}
\caption{One of the $\Delta$-cliques of the RSPLS.}\label{fig:deltaClique}
\end{figure}

\section{Stability of the network}\label{sec:stability}

In this section we keep to the notation of  \cite{AfrMosYou} and 
 show that, for most 
 parameter values in \cite{PosRuc} the RSPLS network is asymptotically stable.
 We start by finding a set that attracts all trajectories that do not start at the origin.
 
 \begin{lemma}\label{lem:InvariantSphere}
 If $\tau_j>0$,  $j=1,\ldots,n$, then  \eqref{eq:RSPLS}  admits a flow-invariant globally attracting $(n-1)$-sphere.
 \end{lemma}
 \begin{proof}
 We transform the ODE \eqref{eq:RSPLS} by changing coordinates as $x_i=X_i^2$. We obtain
\begin{equation}\label{eq:RSPLSac}
\dot X_{i}=\frac{X_i}{2} \left(\tau_i - \sum_{j=1}^n\; \rho_{ij}X_j^{2} \right)  \;\; \mbox{ for  } i=1,\hdots, n.
\end{equation}
The equilibria in the network remain on the coordinate axes but the non-zero coordinate is now represented by $\sqrt{\tau_i}$. At each equilibrium $O_j$, the Jacobian matrix is diagonal. The radial eigenvalues  are preserved and the remaining eigenvalues appear divided by 2. They thus satisfy the assumptions in \eqref{eq:3} and \eqref{eq:4}. 

The nonlinear part of \eqref{eq:RSPLS} is contracting and homogeneous of degree 3. Therefore the Invariant Sphere Theorem of Field \cite{Field1989} holds, ensuring the existence of an attracting invariant $(n-1)$-sphere. 
 \end{proof}

It follows from Lemma~\ref{lem:InvariantSphere} that
the radial eigenvalue is negative, since the invariant sphere is attracting. Then,
at each equilibrium the radial eigenvalue does not have to be taken into account
for the stability of the RSPLS network.

Theorem 2.3 in \cite{AfrMosYou} provides sufficient conditions for the asymptotic stability of the RSPLS network. Other than \eqref{eq:3} and \eqref{eq:4}, for each $k$, these are that
\begin{equation}\label{eq:delta}
\frac{\tau_{k+1}}{\rho_{k+1,k}} \leq \frac{\tau_{k+2}}{\rho_{k+2,k}} 
\end{equation}
and
\begin{equation}\label{eq:5}
\max_{i=1,2} \{ \tau_{k+i}-\rho_{k+i,k}\tau_k\} < \min_{j\neq k,k+1,k+2} \{ |\tau_j-\rho_{jk}\tau_k|,\tau_k\}.
\end{equation}
It is a straightforward consequence of
Lemma~\ref{lem:InvariantSphere}
 that condition \eqref{eq:5} can be simplified to
\begin{equation}\label{eq:5a}
\max_{i=1,2} \{ \tau_{k+i}-\rho_{k+i,k}\tau_k\} < \min_{j\neq k,k+1,k+2} \{ |\tau_j-\rho_{jk}\tau_k|\},
\end{equation}
since the radial eigenvalue ceases to play a role.

Conditions for the asymptotic stability of the RSPLS network are the focus of the next result, with the aim of covering the cases treated in \cite{PosRuc} where $e_A=1$. 
 From the previous correspondence \eqref{eq:rho}, we see that $\rho_{j,j+4}=1-e_A=0$ when $e_A=1$ and thus, does not satisfy the restriction imposed in  \cite{AfrMosYou} that $\rho_{jk}>0$. Furthermore, the hypotheses in Theorem 2.3 of  \cite{AfrMosYou} have to be adapted so that the outgoing connections at each node $\xi_j$ are to $\xi_{j+1}$ and $\xi_{j+3}$. Thus conditions \eqref{eq:3}, \eqref{eq:4}, \eqref{eq:delta} and \eqref{eq:5a} become, respectively,
 \begin{equation}\label{eq:3b}
\min_{i=1,3} \{ \tau_{k+i}-\rho_{k+i,k}\tau_k \} > 0,
\end{equation}
\begin{equation}\label{eq:4b}
\tau_j-\rho_{jk}\tau_k < 0, \;\; \mbox{  for    } \; j\neq k,k+1,k+3,
\end{equation}
 \begin{equation}\label{eq:delta_b}
\frac{\tau_{k+3}}{\rho_{k+3,k}} \leq \frac{\tau_{k+1}}{\rho_{k+1,k}}
\end{equation}
and 
\begin{equation}\label{eq:5b}
\max_{i=1,3} \{ \tau_{k+i}-\rho_{k+i,k}\tau_k\} < \min_{j\neq k,k+1,k+3} \{ |\tau_j-\rho_{jk}\tau_k|\}.
\end{equation}

\begin{proposition}\label{prop:stability}
If $0<e_B< e_A<\min\{c_A,c_B\}$ and $e_A\le1$, then the RSPLS network in \cite{PosRuc} is asymptotically stable.
\end{proposition} 

\begin{proof}
We follow the ideas of the proof of  Theorem 2.3 of  \cite{AfrMosYou} and relax their parameter space to obtain the same result when $e_A=1$. The dynamical system describing the RSPLS game in \cite{PosRuc} is 
\begin{equation}\label{eq:Postleth}
\dot x_{i}=x_i
\Big[1-\Big(x_i+(1 +c_A)x_{i+1}+ (1 -e_B)x_{i+2}+(1 +c_B)x_{i+3}+(1 -e_A)x_{i+4} \Big)   \Big]  .
\end{equation}
As previously stated, $\tau_i=1$ and the remaining coefficients are given in \eqref{eq:rho}. Conditions \eqref{eq:3b} and \eqref{eq:4b} are trivially satisfied. Condition \eqref{eq:5b} leads to $\max \{e_A,e_B\} < \min \{c_A,c_B\}$, implied by our hypothesis. 
Condition \eqref{eq:delta_b} reads as $\dfrac{1}{1-e_B}\le \dfrac{1}{1-e_A}$.
If  $e_A\ne 1\ne e_B$, this is implied by $0<e_B< e_A$. It follows from direct application of the sequence of Lemmas 3.2--3.7 in \cite{AfrMosYou} that the hypotheses of their Theorem 2.3 hold and the network is asymptotically stable.

When $e_A=1$, condition \eqref{eq:delta_b} cannot be verified. We prove that, in this case, the unstable manifold of each equilibrium in contained in the heteroclinic network. We consider the $\Delta$-clique defined by the equilibria $\xi_1$, $\xi_2$ and $\xi_4$ and show that the 2-dimensional unstable manifold of $\xi_1$ is contained in the $\Delta$-clique, which in fact is a $\Delta$-clique in the more restrictive sense of \cite{PodCasLab2020}. We show that there are no equilibria in the portion of state space defined by $x_i > 0$, $i=1,2,4$ and $x_3=x_5=0$.
Such equilibria, if they exist, are in the intersection of the following three planes
\begin{eqnarray*}
P_1 & = & \{ 1- x_1 -(1+c_A)x_2-(1+c_B)x_4 =0 \} \\
P_2 & = & \{ 1- (1-e_A)x_1-x_2 - (1-e_B)x_4 =0 \} \\
P_4& = & \{ 1- (1-e_B)x_1-(1+c_B)x_2-x_4 =0 \} 
\end{eqnarray*}
To see that the planes $P_1$, $P_2$ and $P_4$ do not intersect in the interior of the $\Delta$-clique we show that $P_2$ {\em dominates} the other two planes. We say, as in \cite{AfrMosYou}, that the plane $P_2$ dominates $P_1$ if, when representing each plane by the graph of a function $x_2=z_2(x_1,x_4)$ and $x_1=z_1(x_2,x_4)$, the graph representing $P_2$ is always above that representing $P_1$. Analogously, for the statement that $P_2$ dominates $P_4$.
The intersections of  the planes with the coordinate axes are as follows, when $e_A=1$:
\begin{center}
\begin{tabular}{c|lll}
plane$\backslash$axis	&	$x_1$	&	$x_2$	&	$x_4$	\\ \hline
$P_1$	&	$x_1=1$	&	$x_2=1/(1+c_A)$\quad	&	$x_4=1/(1+c_B)$	\\
$P_2$	&	$\varnothing$\quad	&	$x_2=1$	&	$x_4=1/(1-e_B)$	\quad	\\
$P_4$	&	$x_1=1/(1-e_B)$	&	$x_2=1/(1+c_B)$	&	$x_4=1$
\end{tabular}
\end{center}
Since $1/(1+c_A), 1/(1+c_B) <1$ and $1/(1-e_B)>1$, it is easy to see that the intersection of $P_2$ with the axes $x_2$ and $x_4$ is larger than those of either $P_1$ or $P_4$.
Hence, $P_2$ is always above the other two planes for $x_i > 0$, $i=1,2,4$ and $x_3=x_5=0$. 
The proof that the $\Delta$-clique exists follows analogously to the case $e_A \neq 1$. Note that we are working in an attracting invariant topological sphere in three-dimensional space so that trajectories do not go to infinity. Since the invariant sphere is compact and 2-dimensional, Poincar\'e-Bendixson requires an equilibrium for the existence of a periodic orbit. Since there are no equilibria, no period orbits exist.

The proof for the remaining $\Delta$-cliques in the network is done by permutation of the indices.
\end{proof}

The next result establishes the asymptotic stability of the RSPLS network for most values in Figure~7 of \cite{PosRuc} that correspond to the existence of `sausages'\footnote{The term `sausage' has been used by the authors of \cite{PosRuc} to describe small intertwined regions in parameter space with different dynamics and depicted in their Figures~2, 7 and 8.}. These correspond to fragmentary asymptotic stability regions for various sequences other than the Rock-to-Spock, the Star and the RSP cycles referred to above. 
Establishing the asymptotic stability of the whole network supports the visibility of the sausages of \cite{PosRuc}.
Note that for some portion of the region depicted in \cite{PosRuc}, namely $c_A \in (0.8,1]$, the sufficient conditions for asymptotic stability of the network given in \cite{AfrMosYou} do not apply.

\begin{corollary}\label{cor:as-stab}
The RSPLS network in \cite{PosRuc} is asymptotically stable if
$e_A=1$, $e_B=0.8$, $c_A \in (1.0,1.8)$
and $c_B \in (1,4.5)$.
\end{corollary}

\section{Stability of the four elementary cycles}\label{sec:stab-cycles}

In this section we present the stability indices for the sub-cycles of 1-dimensional heteroclinic connections of the 
four cycles: Rock-to-Paper, Star, RSP and Four-node.
From now on we remain with the formulation of  \cite{PosRuc} given in \eqref{eq:Postleth}.
Recall the relation between our cycles and those of \cite{PosRuc}: our Rock-to-Paper cycle is of type A, our Star cycle is of type B, 
our RSP cycle is of type AAB, and our Four-node cycle is of type $\text{Q}=\text{ABBB}$.

\subsection{Previous results} 
As is shown in~\cite{GarCas2019} the stability indices can be calculated for the general class of quasi-simple cycles.
It is easily seen that the 
four cycles of interest are either quasi-simple or have quasi-simple sub-cycles when restricted to the flow-invariant coordinate planes. 
This restriction ensures that Definition~\ref{def:quasi} is satisfied since, for these sub-cycles, all invariant $P_j$'s are coordinate planes and dim$(P_j \ominus \hat{L}_j)=1$.
Actually, the sub-cycles so obtained admit at every equilibrium one radial, one contracting, one expanding and two transverse eigenvalues. 
We refer the reader to \cite{GarCas2019} for detail on the classification of the eigenvalues.
All the connections in the Star cycle are one-dimensional,  in this case the sub-cycle coincides with the cycle.
We label the sub-cycles as $\Sigma_\text{R-to-P}$, $\Sigma_\text{Star}$, $\Sigma_\text{RSP}$
and $\Sigma_\text{4-node}$, respectively.

For every $j=1,\ldots,5$, the eigenvalues of $\xi_j$ are $-1$, $e_A$, $-c_B$, $e_B$ and $-c_A$, with eigenvectors in the $x_j$, $x_{j+1}$, $x_{j+2}$, $x_{j+3}$ and $x_{j+4}$ directions $(\text{mod }5)$, respectively. 
This naturally adds symmetry to the problem under the action of the group $\Z_5(\varphi)$ with $\varphi(x_1,x_2,x_3,x_4,x_5)=(x_5,x_1,x_2,x_3,x_4)$ as in~\cite{PosRuc}.
Let $H^{\inn}_j$ stand for the cross-section to the flow at an incoming connection to~$\xi_j$. Since the radial direction can be omitted all cross-sections are 3-dimensional 
-- we  take cross-sections within the invariant 4-sphere.
The dynamics near each sub-cycle is approximated by basic transition matrices\footnote{A basic transition matrix provides a convenient description of the dynamics from one incoming cross-section to the next. Its entries are 0's and 1's, except for 
one column which consists of quotients between the modulus of the contracting and expanding eigenvalues (for one entry) and between the symmetric of transverse eigenvalues and the expanding eigenvalue (for the remaining rows). The definition of a transition matrix goes back to the work of Field and Swift \cite{FieldSwift}. A detailed construction of basic transition matrices in the context of cycles of type Z (a subset of quasi-simple cycles) can be found in \cite{Podvigina2012}. Here we use the work of \cite{GarCas2019}.} 
$\M_j:H^{\inn}_j \rightarrow H^{\inn}_{j+1}$ whose entries are rational functions of the eigenvalues at~$\xi_j$,
where we change the indexing, so now $\xi_{j+1}$ is the equilibrium with a connection $\xi_j\to\xi_{j+1}$ in the sub-cycle under study.
The basic transition matrices coincide with those presented in \cite[Subsection 4.1]{PosRuc}. The results from~\cite{GarCas2019} hold in the present case and the stability of the sub-cycles is governed by properties of the basic transition matrices and their product:
\begin{align*}
\M^{(j)}: H^{\inn}_j & \rightarrow H^{\inn}_j,  &  \M^{(j)} & = \M_{j-1} \hdots \M_1 \M_m \hdots \M_j, \\[0.2cm]
\M_{\left(l,j\right)}: H^{\inn}_j  &\rightarrow H^{\inn}_{l+1},  & \M_{\left(l,j\right)}  & =\begin{cases}
\M_{l}\ldots \M_{j}, & l>j\\
\M_{l}\ldots \M_{1}\M_{m}\ldots \M_{j}, & l<j\\
\M_{j}, & l=j,
\end{cases}
\end{align*}
where 
$m \in \{3,4,5\}$ is the number of equilibria. 

Given a $3\times 3$ 
matrix $\M$, denote by $\lambda_{\max}$ the maximal eigenvalue in absolute value and by $\boldsymbol{w}^{\max}=(w^{\max}_1,w^{\max}_2, w^{\max}_3)^\text{T}$ the corresponding eigenvector, where the superscript ``T'' indicates the transpose of a matrix in general.
The conditions for stability are (cf~\cite[Lemma 3.2]{GarCas2019}):
\begin{enumerate}
\item[\emph{(i)}]
$\lambda_{\max}$ is real,
\item[\emph{(ii)}]
$\lambda_{\max}>1$,
\item[\emph{(iii)}]
$w^{\max}_l w^{\max}_q >0$ for all $l,q = 1,2,3$. 
\end{enumerate}
Combining these with Theorem~3.10 in~\cite{GarCas2019} we derive expressions for the stability indices by means of a function $F^{\ind}$. 
We reproduce the values of $F^{\ind}(\boldsymbol{\alpha})$ for any $\boldsymbol{\alpha}=(\alpha_1,\alpha_2,\alpha_3) \in \R^3$ from Appendix~A.1 of~\cite{GarCas2019}:  
\[
F^{\ind}\left(\boldsymbol{\alpha}\right)=\begin{cases}
+\infty, & \textrm{if }\min\left\{ \alpha_{1},\alpha_{2},\alpha_{3}\right\} \geq0\\[0.1cm]
-\infty, & \textrm{if }\max\left\{ \alpha_{1},\alpha_{2},\alpha_{3}\right\} \leq0\\[0.1cm]
0, & \textrm{if }\alpha_{1}+\alpha_{2}+\alpha_{3}=0\\[0.1cm]
\dfrac{\alpha_{1}+\alpha_{2}+\alpha_{3}}{\max\left\{ \alpha_{1},\alpha_{2},\alpha_{3}\right\} }, & \textrm{if }\max\left\{ \alpha_{1},\alpha_{2},\alpha_{3}\right\} >0\textrm{ and }\alpha_{1}+\alpha_{2}+\alpha_{3}<0\\[0.5cm]
-\dfrac{\alpha_{1}+\alpha_{2}+\alpha_{3}}{\min\left\{ \alpha_{1},\alpha_{2},\alpha_{3}\right\} }, & \textrm{if }\min\left\{ \alpha_{1},\alpha_{2},\alpha_{3}\right\} <0\textrm{ and }\alpha_{1}+\alpha_{2}+\alpha_{3}>0.
\end{cases}
\]

In virtue of one repelling transverse direction at every $\xi_j$, all basic transition matrices~$\M_j$ have one negative entry. 
Define $\sigma_j$ to be the stability index along the incoming connection to $\xi_j$. 
The following proposition adapts Theorem 3.10 in~\cite{GarCas2019} to our setting which naturally satisfies Assumption 3.1 in~\cite{GarCas2019} (the global maps are described by permutation matrices).

\begin{proposition}[Theorem 3.10 in~\cite{GarCas2019}]
\label{prop:calc-index}
Let $\Sigma$ be a quasi-simple cycle with  basic transition matrices $\M_j$, $j = 1, \ldots, m$. 
\begin{enumerate}
\renewcommand{\theenumi}{(\alph{enumi})}
\renewcommand{\labelenumi}{{\theenumi}}
\item
If $\M^{(j)}$ does not satisfy conditions (i)--(iii)  for at least one $j$, then $\sigma_j= -\infty$ for all~$j$ and $\Sigma$ is completely unstable. 
\item
If $\M^{(j)}$ satisfies conditions (i)--(iii) for all~$j$, then $\Sigma$ is f.a.s. and there exist vectors $\boldsymbol{\alpha}^{(1)},\boldsymbol{\alpha}^{(2)},...,\boldsymbol{\alpha}^{(K)} \in \R^3$ such that\[
\sigma_j = \min_{i=1,\ldots,K} \left\{ F^{\ind} \left( \boldsymbol{\alpha}^{(i)}\right) \right\}.
\]
\end{enumerate}
\end{proposition}

For each $j=1,\ldots,m$, the vectors $\boldsymbol{\alpha}^{(i)}$ that must be considered are the rows of the transition matrices $\M_{(j,j)}=\M_j$, $\M_{(j+1,j)}=\M_{j+1}\M_j$, $\M_{(j+2,j)}=\M_{j+2}\M_{j+1}\M_j$, $\ldots,$ $\M_{(j-1,j)}=\M^{(j)}$. The number~$K$ refers to the number of such rows whenever
\begin{equation}
\label{eq:U-infty}
U^{-\infty}\left(\M^{(j)}\right) = \left\{\boldsymbol{y} \in \R_-^3: \ \lim_{k\rightarrow +\infty} \left(\M^{(j)} \right)^k \boldsymbol{y} = \boldsymbol{-\infty} \right\} = \R_-^3,
\end{equation}
is satisfied, where $\R_-^3=\left\{ \boldsymbol{y}=(y_1,y_2,y_3) \in \R^3: \ y_1,y_2,y_3<0\right\}$, see \cite{GarCas2019} for details.

\subsection{Stability of the elementary cycles}

In this subsection we provide the stability results for each of the four elementary cycles in the RSPLS network.

\paragraph{The Rock-to-Paper sub-cycle:} 
The Rock-to-Paper sub-cycle comprises five equilibria and five 1-dimensional heteroclinic connections in the order, see Figure~\ref{fig:cycles}(a):
\[
\Sigma_\text{R-to-P}=[\xi_1 \rightarrow \xi_2 \rightarrow \xi_3 \rightarrow \xi_4 \rightarrow \xi_5 \rightarrow \xi_1].
\]
The behaviour of trajectories between any two consecutive equilibria is captured up to a permutation by the basic transition matrix $\M_2 : H_2^{\inn} \rightarrow H_3^{\inn}$ with
\[
\M_2 = \begin{bmatrix}
\dfrac{c_B}{e_A} & 0 & 1 \\[0.5cm]
\dfrac{c_A}{e_A} & 0 & 0 \\[0.5cm]
-\dfrac{e_B}{e_A} & 1 & 0
\end{bmatrix}.
\]
Starting near each equilibrium, 
  the powers $(\M_2)^l$, $l=1,\ldots,5$ provide an approximation of a trajectory that visits once a neighbourhood of  each equilibrium of $\Sigma_\text{R-to-P}$.
The stability indices may thus be computed from the rows of $\M_2$.

We have the following:

\begin{proposition} \label{prop:ind-R-to-P}
The local stability indices for the Rock-to-Paper sub-cycle $\Sigma_\text{R-to-P}$ are all equal and:
\begin{enumerate}
\renewcommand{\theenumi}{(\alph{enumi})}
\renewcommand{\labelenumi}{{\theenumi}}
\item
if either $c_A + c_B < e_A + e_B$ or $c_A e_A < c_B e_B$ or $c_A c_B^3 < e_A e_B^3$, then $\sigma_\text{R-to-P} = -\infty$.
\item
if $c_A + c_B > e_A + e_B$ and $c_A e_A > c_B e_B$ and $c_A c_B^3 > e_A e_B^3$, then 
\[
 -\infty <  \ 
\sigma_\text{R-to-P} 
\ \leq \
F^{\ind}\left( -\frac{e_B}{e_A},1,0\right).
\]
\end{enumerate}
\end{proposition}

\begin{proof}
According to Proposition~\ref{prop:calc-index}, the stability of $\Sigma_\text{R-to-P}$ depends on whether or not~$\M_2$ satisfies conditions \emph{(i)}--\emph{(iii)}.
Eigenvalues of $\M_2$ are the roots of the characteristic polynomial
\[
p(\lambda)=-\lambda^3 + a_2 \lambda^2 + a_3 \lambda + a_1,
\]
where
\[
a_1 = \frac{c_A}{e_A}, \quad a_2 =  \frac{c_B}{e_A}, \quad a_3 = -\frac{e_B}{e_A}.
\]
Let $\lambda_1,\lambda_2, \lambda_3 \in \C$ be the eigenvalues of $\M_2$ such that $\lambda_1 = \lambda_{\max}$ and $\boldsymbol{w}^{\max}$ is the eigenvector associated with $\lambda_1$. Vieta's formulas applied to cubic polynomials give
\begin{equation}
\label{eq:viete}
\lambda_1 + \lambda_2 + \lambda_3 = \frac{c_B}{e_A}, \qquad 
\lambda_1\lambda_2 + \lambda_1\lambda_3 + \lambda_2 \lambda_3  = \frac{e_B}{e_A}, \qquad 
\lambda_1 \lambda_2 \lambda_3  = \frac{c_A}{e_A}.
\end{equation}
Using Lemma~10 in~\cite{Podvigina2013} we find that conditions \emph{(i)}--\emph{(iii)} for $\M_2$ are individually fulfilled if and only if 
\begin{align}
\frac{c_A}{e_A} +  \frac{c_B}{e_A}  -\frac{e_B}{e_A} >1 \ & \Leftrightarrow  \ c_A + c_B > e_A + e_B, \label{eq:cond-i} \\[0.2cm]
- \frac{c_B}{e_A} \frac{e_B}{e_A} +  \frac{c_A}{e_A} > 0 \ & \Leftrightarrow  \ c_A e_A > c_B e_B, \label{eq:cond-ii}  \\[0.2cm]
\frac{c_A}{e_A}\frac{c_B^3}{e_A^3}  -\frac{e_B^3}{e_A^3} >0 \ & \Leftrightarrow \ c_A c_B^3 > e_A e_B^3. \label{eq:cond-iii} 
\end{align} 

When one of the relations~\eqref{eq:cond-i} to \eqref{eq:cond-iii} does not hold, statement~(a) is immediate from Proposition~\ref{prop:calc-index}(a). 

Suppose now that~\eqref{eq:cond-i}--\eqref{eq:cond-iii} hold true. Then, $\lambda_1=\lambda_{\max}>1$ and the components of $\boldsymbol{w}^{\max}$ have all the same sign. The identities in~\eqref{eq:viete} enable one to disclose that $\lambda_1$ is the only eigenvalue with positive real part. It follows that $\lambda_2+\lambda_3<0$ and $\lambda_2\lambda_3>0$.
We check that $U^{-\infty}\left(\M_2\right) =\R_-^3$ in~\eqref{eq:U-infty}.
This is equivalent to showing that any $\boldsymbol{y} \in \R_-^3$ written in the eigenbasis of $\M_2$ must have a negative coefficient for the largest eigenvector. The coefficient writes as
$
( \boldsymbol{v}^{\max})^\text{T} \boldsymbol{y},
$
where $\boldsymbol{v}^{\max}$ is a vector multiple of
\[
\left( \lambda_1 ,  \frac{1}{\lambda_1} ,  1 \right)^\text{T}.
\]
Because $\lambda_1 > 0$ we get $( \boldsymbol{v}^{\max})^\text{T} \boldsymbol{y}<0$ for any $\boldsymbol{y} \in \R_-^3$. 
Proposition~\ref{prop:calc-index}(b) applies and we need to take into account the rows with at least one negative entry of $\M_2$, $(\M_2)^2$, $(\M_2)^3$, $(\M_2)^4$ and $(\M_2)^5$ in Appendix~\ref{app:R-to-P}, so that 
\[
\begin{aligned}
\sigma_\text{R-to-P}
\leq & \ F^\textnormal{index} \left(-\frac{e_B}{e_A},1,0\right).
\end{aligned}
\]
\end{proof}

\paragraph{The Star cycle:}
The Star cycle comprises five equilibria and five 1-dimensional heteroclinic connections in the order, see Figure~\ref{fig:cycles}(b):
\[
\Sigma_\text{Star}=[\xi_1 \rightarrow \xi_4 \rightarrow \xi_2 \rightarrow \xi_5 \rightarrow \xi_3 \rightarrow \xi_1].
\]
The transition between any two consecutive equilibria is described up to a permutation by the basic transition matrix $\M_4 : H_4^{\inn} \rightarrow H_2^{\inn}$ with
\[
\M_4= \begin{bmatrix}
0 & \dfrac{c_A}{e_B} & 1 \\[0.5cm]
1 & -\dfrac{e_A}{e_B} & 0 \\[0.5cm]
0 & \dfrac{c_B}{e_B} & 0
\end{bmatrix}.
\]
Again we are reduced to establishing the stability properties of $\Sigma_\text{Star}$ by taking the rows of~$\M_4$ as follows.
\begin{proposition}\label{prop:ind-Star}
The local stability indices for the Star cycle 
$\Sigma_\text{Star}$ are all equal and:
\begin{enumerate}
\renewcommand{\theenumi}{(\alph{enumi})}
\renewcommand{\labelenumi}{{\theenumi}}
\item
if either $c_A + c_B < e_A + e_B$ or $c_B e_B < c_A e_A$ or $c_A^3 e_B < c_B e_A^3$, then $\sigma_\text{Star} = -\infty$.
\item
if $c_A + c_B > e_A + e_B$ and $c_B e_B > c_A e_A$ and $c_A^3 e_B > c_B e_A^3$, then 
\[
- \infty < \
\sigma_\text{Star} 
\ \leq \
F^{\ind}\left( 1, -\frac{e_A}{e_B}, 0\right).
\]
\end{enumerate}
\end{proposition}

\begin{proof}
This follows by the same method of the proof of Proposition~\ref{prop:ind-R-to-P},
making use of the transition matrix $\M_4$ whose characteristic polynomial is $p(\lambda)= -\lambda^3 - \frac{e_A}{e_B} \lambda^2+\frac{c_A}{e_B} \lambda + \frac{c_B}{e_B}$. 
For~(b) the negative entries of $\M_4$, $(\M_4)^2$, $(\M_4)^3$, $(\M_4)^4$ and $(\M_4)^5$ in Appendix~\ref{app:Star} must be considered. 
\end{proof}

\paragraph{The RSP sub-cycle:}
The RSP sub-cycle comprises three equilibria and three 1-dimensional connections in the order, see Figure~\ref{fig:cycles}(c):
 
 \[
\Sigma_\text{RSP}=[\xi_1 \rightarrow \xi_2 \rightarrow \xi_3 \rightarrow \xi_1].
\]
We write down the three basic transition matrices  
$\M_j : H_j^{\inn} \rightarrow H_{j+1}^{\inn}$, $j= 1,2,3  \text{ (mod } 3)$, with
\[
\M_1 = \begin{bmatrix}
\dfrac{c_B}{e_A} & 0 & 0 \\[0.5cm]
\dfrac{c_A}{e_A} & 0 & 1 \\[0.5cm]
-\dfrac{e_B}{e_A} & 1 & 0
\end{bmatrix}, 
\quad \M_2 = \begin{bmatrix}
\dfrac{c_B}{e_A} & 0 & 1 \\[0.5cm]
\dfrac{c_A}{e_A} & 0 & 0 \\[0.5cm]
-\dfrac{e_B}{e_A} & 1 & 0
\end{bmatrix},
\quad \M_3  = \begin{bmatrix}
0 &\dfrac{c_A}{e_B} & 0  \\[0.5cm]
1 & -\dfrac{e_A}{e_B} & 0 \\[0.5cm]
0 & \dfrac{c_B}{e_B} & 1  
\end{bmatrix}.
\]
The products $\M_{(j+1,j)}= \M_{j+1} \M_j: H_j^{\inn} \rightarrow H_{j+2}^{\inn}$ and $\M^{(j)}=\M_{j+2} \M_{j+1} \M_j: H_j^{\inn}  \rightarrow H_j^{\inn}$, $j= 1,2,3  \text{ (mod } 3)$ can be found in Appendix~\ref{subsec:RSP}. 
The following quantities are useful:
\begin{align*}
\delta_T & = \frac{c_A^2 c_B}{e_A^2 e_B} &  
\gamma_T & = \frac{c_A^3}{e_A^2 e_B} + \frac{c_B c_A}{e_B e_A} - \frac{e_B}{e_A} \\[0.2cm]
\alpha_T & = \frac{c_B^2}{e_A^2} - \frac{c_A c_B}{e_B e_A}- \frac{e_B}{e_A}  &
\theta_T & = - \frac{c_A^2}{e_A^2} +  \frac{c_B}{e_A} - \frac{c_A}{e_B}
\\[0.2cm]
\beta_T & =  \frac{c_B^2 c_A}{e_A^2 e_B} - \frac{e_B c_B}{e_A^2} + \frac{c_A}{e_A}  &
\mu_T & = \frac{c_B^2 c_A}{e_A^2 e_B} - \frac{c_A}{e_A} - \frac{e_A}{e_B} \\[0.2cm]
 & & 
 \nu_T & =  -\frac{c_B c_A}{e_A^2} + \frac{c_A^2}{e_A e_B} + \frac{c_B}{e_B}.
\end{align*}

In the next result the conditions imposed in \ref{RSPa} and \ref{RSPb} are complementary, in view of Lemma~\ref{lem:entries} in Appendix~\ref{subsec:proof}.
We denote by $\sigma_{ij}$ the stability index along the trajectory connecting $\xi_i$ to $\xi_j$.

\begin{proposition} \label{prop:index-RSP}
The local stability indices for the RSP sub-cycle $\Sigma_\text{RSP}$ are:
\begin{enumerate}
\renewcommand{\theenumi}{(\alph{enumi})}
\renewcommand{\labelenumi}{{\theenumi}}
\item\label{RSPa}
 if either $\delta_T<1$ or 
 $\alpha_T<0$ or $\beta_T<0$ or 
 $\gamma_T<0$ or 
$\theta_T<0$ or 
 $\mu_T<0$,
 or $\nu_T<0$,
   then $\sigma_{31}= \sigma_{12} = \sigma_{23} = -\infty$.
\item\label{RSPb}
 if $\delta_T > 1$ and $\theta_T>0$ and $\nu_T>0$, then
\begin{align*}
\sigma_{31} &  = \begin{cases}
\min \left\{ 1 - \dfrac{e_B}{e_A},  1 - \dfrac{e_Bc_B}{e_A^2} + \dfrac{c_A}{e_A}\right\}\; \;  (<0) & \text{ if } \dfrac{e_B}{e_A} > \max\left\{1, \dfrac{e_A+c_A}{c_B}\right\} \\[0.5cm]
1 - \dfrac{e_B}{e_A} \; \;  (<0) & \text{ if }  1 < \dfrac{e_B}{e_A} < \dfrac{e_A+c_A}{c_B} \\[0.5cm]
1 - \dfrac{e_Bc_B}{e_A^2} + \dfrac{c_A}{e_A} \; \;  (<0) & \text{ if }   \dfrac{e_A+c_A}{c_B}  \leq \dfrac{e_B}{e_A} < 1 \\[0.5cm]
\min \left\{ -1 + \dfrac{e_A}{e_B},  -1 + \dfrac{e_A^2}{e_B c_B - c_A e_A }\right\}\; \;  (>0) & \text{ if } \dfrac{e_B}{e_A} < \min\left\{1, \dfrac{e_A+c_A}{c_B}\right\} 
\end{cases} \\[0.5cm]
\sigma_{12}  & = \begin{cases}
1 - \dfrac{e_B}{e_A}   \; \;  (<0) & \text{ if }  \begin{aligned}[t] 1 < \dfrac{e_B}{e_A} \leq \dfrac{c_Bc_A}{e_Be_A} & \text{ or } \max\left\{1,\dfrac{c_Bc_A}{e_Be_A} \right\} < \dfrac{e_B}{e_A} < 1+ \dfrac{c_Bc_A}{e_Be_A} \\[0.3cm]
\text{ or } &\dfrac{e_B}{e_A} \geq 1+ \dfrac{c_Bc_A}{e_B e_A} \end{aligned} \\[1.5cm]
 -1 + \dfrac{e_A}{e_B}   \; \;  (>0)  & \text{ if } \dfrac{e_B}{e_A} < \min\left\{1, \dfrac{c_Bc_A}{e_Be_A} \right\} \text{ or }  \dfrac{c_Bc_A}{e_Be_A} < \dfrac{e_B}{e_A} < 1
 \end{cases} \\[0.5cm]
\sigma_{23} & = \begin{cases}
1- \dfrac{c_A}{e_A}  -  \dfrac{e_A}{e_B}   \; \;  (<0) & \text{ if } \dfrac{e_B}{e_A}<1 \text{ or } 1<  \dfrac{e_B}{e_A} <  \dfrac{e_A}{e_A-c_A}\\[0.5cm]
1-\dfrac{e_A}{e_B}  \; \;  (<0) & \text{ if } \dfrac{e_B}{e_A} > \dfrac{e_A}{e_A-c_A}.
 \end{cases} 
\end{align*}
\end{enumerate}
\end{proposition}

\begin{proof}
This follows by the same method of the proof of Proposition~\ref{prop:ind-R-to-P}, see Appendix~\ref{subsec:proof} where Lemma~\ref{lem:entries} explains why some of the quantities do not appear in statement (b).
\end{proof} 

\paragraph{The Four-node sub-cycle:}
The Four-node sub-cycle comprises four equilibria
and four 1-dimensional heteroclinic connections 
in the order, see Figure~\ref{fig:cycles}(d):

\[
\Sigma_\text{4-node}=[\xi_1 \rightarrow \xi_2 \rightarrow \xi_5 \rightarrow \xi_3 \rightarrow \xi_1].
\]
The four basic transition matrices between consecutive equilibria are\footnote{Based on the type of heteroclinic connection, we have the following correspondence: $\widehat{\M}_1 = \M_1$, $\widehat{\M}_2 = \M_3$, $\widehat{\M}_3 = \widehat{\M}_5 = \M_4$.}
\[
\widehat{M}_1 :\widehat{H}_1^{\inn} \rightarrow \widehat{H}_2^{\inn}, \qquad \widehat{M}_2 :\widehat{H}_2^{\inn} \rightarrow \widehat{H}_5^{\inn}, \qquad \widehat{M}_5 :\widehat{H}_5^{\inn} \rightarrow \widehat{H}_3^{\inn}, \qquad \widehat{M}_3 :\widehat{H}_3^{\inn} \rightarrow \widehat{H}_1^{\inn}
\]
where
\[
\widehat{\M}_1 = \begin{bmatrix}
\dfrac{c_B}{e_A} & 0 & 0 \\[0.5cm]
\dfrac{c_A}{e_A} & 0 & 1 \\[0.5cm]
-\dfrac{e_B}{e_A} & 1 & 0
\end{bmatrix}, 
\quad \widehat{\M}_2 = \begin{bmatrix}
0 &\dfrac{c_A}{e_B} & 0  \\[0.5cm]
1 & -\dfrac{e_A}{e_B} & 0 \\[0.5cm]
0 & \dfrac{c_B}{e_B} & 1  
\end{bmatrix},
\quad   \widehat{\M}_3 = \widehat{\M}_5 = \begin{bmatrix}
0 & \dfrac{c_A}{e_B} & 1 \\[0.5cm]
1 & -\dfrac{e_A}{e_B} & 0 \\[0.5cm]
0 & \dfrac{c_B}{e_B} & 0
\end{bmatrix}. 
\]
The products of the basic transition matrices with respect to the Four-node sub-cycle near~$\xi_1$ are
$\widehat{\M}_{(2,1)} =  \widehat{\M}_2 \widehat{\M}_1    :  \widehat{H}_1^{\inn} \rightarrow \widehat{H}_{5}^{\inn}$, $\widehat{\M}_{(5,1)}  = \widehat{\M}_5 \widehat{\M}_2 \widehat{\M}_1    :  \widehat{H}_1^{\inn} \rightarrow \widehat{H}_{3}^{\inn}$ and $\widehat{\M}^{(1)}  =  \widehat{\M}_3 \widehat{\M}_5 \widehat{\M}_2 \widehat{\M}_1   :   \widehat{H}_1^{\inn}  \rightarrow \widehat{H}_1^{\inn}$.
In the same manner we obtain the products near $\xi_j$, $j=2,3,5$. All transition matrix products can be found in Appendix~\ref{app:4node}.

The next result makes use of notation introduced in Appendix~\ref{subsec:proof4nodes}.
\begin{proposition} \label{prop:index-4node}
The local stability indices for the Four-node sub-cycle $\Sigma_\text{4-node}$ are:
\begin{enumerate}
\item[(a)] if either $\alpha_{11} + \alpha_{33} < \min \left\{ 2, 1 + \frac{c_B^3 c_A}{e_B^3 e_A}\right\}$ or 
$\theta_T > 0$ or $\nu_T > 0$
or $w_2^{\max,1}<0$ or $w_3^{\max,2}<0$ or $w_3^{\max,5}<0$ or $w_1^{\max,3}<0$, then $\sigma_{31} = \sigma_{12} =\sigma_{25} = \sigma_{53} = -\infty$.
\item[(b)] if $\alpha_{11} + \alpha_{33} > \min \left\{ 2, 1 + \frac{c_B^3 c_A}{e_B^3 e_A}\right\}$ and 
$\theta_T < 0$ and 
$\nu_T < 0$
and $w_3^{\max,2}>0$, then
\begin{align*}
 -\infty <  \sigma_{31} & \leq F^\textnormal{index}\left(-\frac{e_B}{e_A},1,0 \right) \\
 -\infty <  \sigma_{12},  \  \sigma_{25}, \  \sigma_{53} &  \leq F^\textnormal{index}\left(1,-\frac{e_A}{e_B},0 \right). 
\end{align*}
\end{enumerate}
\end{proposition}

\begin{proof}
This follows by the same method of the proof of Proposition~\ref{prop:ind-R-to-P}, see Appendix~\ref{subsec:proof4nodes}.
\end{proof} 

It is clear that the Four-node cycle cannot be e.a.s. In fact, 
$$
F^\textnormal{index}\left(-\frac{e_B}{e_A},1,0 \right) \cdot F^\textnormal{index}\left(1,-\frac{e_A}{e_B},0 \right) = -\left(1-\frac{e_A}{e_B}\right)^{2} < 0
$$ 
so that the stability indices cannot all be positive.

Regardless of the stability exhibited by the heteroclinic network as a whole, as expected, not all cycles can be simultaneously stable.

\begin{lemma}\label{lem:no2stable}
At most one of the 5-node cycles in the RSPLS network is f.a.s.
At most either the 3-node sub-cycle or the 4-node sub-cycle is f.a.s.
Furthermore, if the sub-cycle $\Sigma_\text{R-to-P}$ is f.a.s. then the sub-cycle $\Sigma_\text{4-node}$ is c.u.
\end{lemma}

\begin{proof}
The stability is obtained by using Lemma 2.5 in \cite{GarCas2019} to relate a finite stability index to f.a.s.
It is clear from Propositions~\ref{prop:ind-R-to-P} and \ref{prop:ind-Star} that the sufficient condition $c_A e_A < c_B e_B$ for the sub-cycle $\Sigma_\text{R-to-P}$ to be c.u.\ is satisfied when the cycle $\Sigma_\text{Star}$ has a finite stability index. Analogously, the sufficient condition $c_B e_B < c_A e_A$ for the cycle $\Sigma_\text{Star}$ to be c.u.\ is satisfied when the sub-cycle $\Sigma_\text{R-to-P}$ has a finite stability index.

The conditions on the sign of $\theta_T$ and $\nu_T$ for the cycles $\Sigma_\text{RSP}$ and $\Sigma_\text{4-node}$ are exclusive. Hence, at most one of these two cycles is f.a.s.

The condition $c_B e_B < c_A e_A$ that is satisfied if $\Sigma_\text{R-to-P}$ is f.a.s. guarantees that $\nu_T>0$ and therefore $\Sigma_\text{4-node}$ is c.u.
\end{proof}

\subsection{Stability of cycles in an asymptotically stable network}\label{sec:net_as}

From now on, we assume that the sufficient condition for the asymptotic stability (a.s.) of the whole heteroclinic network holds.
From Proposition~\ref{prop:stability} we restrict the parameter space to 
\begin{equation}\label{hyp:as}
0<e_B< e_A<\min\{c_A,c_B\} \quad  \mbox{  and    } \quad e_A\le1.
\end{equation}

\begin{lemma}
\label{lem:relations_as}
The following relations hold:
\begin{enumerate}
\renewcommand{\theenumi}{\alph{enumi}}
\renewcommand{\labelenumi}{({\theenumi})}
\item \label{item:16satisfied}
if \eqref{hyp:as} is satisfied, then
\begin{enumerate}
\renewcommand{\theenumii}{(\theenumi\arabic{enumii})}
\renewcommand{\labelenumii}{{\theenumii}}
	\item $\delta_T >1$
	, $\beta_T>0$ 
	and $\gamma_T>0$; 
	\item $c_A+c_B > e_A+e_B$; 
	\item $c_Ac_B^3 > e_Ae_B^3$. 
\end{enumerate}
\item if either 
$\alpha_T>0$ or 
$\theta_T>0$ or 
$\beta_T<0$ or 
$\nu_T<0$, then $c_A e_A < c_B e_B$. 
\item if $\theta_T>0$ and $\nu_T>0$, then $c_A^3e_B < c_B e_A^3$. Otherwise, if $\theta_T<0$ and $\nu_T<0$, then $c_A^3e_B > c_B e_A^3$.
\end{enumerate}
\end{lemma}
\begin{proof}
In (a), given \eqref{hyp:as}, it is immediate that $\delta_T>1$, $c_A+c_B > e_A + e_B$ and $c_A c_B^3 > e_A e_B^3$. Write

\begin{align*}
\beta_T = \frac{c_B(c_Bc_A -e_B^2) + c_A e_A e_B}{e_A^2 e_B} &  \quad \text{ and } \quad  \gamma_T = \frac{c_A^3 + e_A ( c_Bc_A -e_B^2)}{e_A^2 e_B}.
\intertext{Now $c_Bc_A - e_B^2>0$ from \eqref{hyp:as} and the signs of $\beta_T$ and $\gamma_T$ are respectively given by}
c_B(c_Bc_A -e_B^2) + c_A e_A e_B>0 &  \quad \text{ and }  \quad  c_A^3 + e_A ( c_Bc_A -e_B^2) >0.
\end{align*}

We establish (b) by expressing 

\begin{align}
\alpha_T & =  \frac{-c_B(c_Ae_A-c_Be_B) -e_B^2e_A}{e_A^2e_B}, &
\theta_T & = \frac{-c_A^2e_B - e_A(c_Ae_A- c_Be_B)}{e_A^2e_B}, \label{eq:ag-T} \\
\beta_T& = \frac{c_B^2c_A + e_B(c_A e_A-c_Be_B)}{e_A^2 e_B}, & 
\nu_T & = \frac{c_A(c_A e_A-c_Be_B) + e_A^2c_B}{e_A^2 e_B}. \label{eq:bn-T}
\end{align}

In (c), observe that
\begin{align*}
\theta_T > 0 & \  \Leftrightarrow \ c_A(-c_A^2e_B+ c_B e_A e_B-e_A^2 c_A)  > 0  &  \  \Leftrightarrow \ c_A c_B e_A e_B - e_A^2 c_A^2 > c_A^3 e_B, \\[0.2cm] 
\nu_T > 0 & \  \Leftrightarrow \ e_A(-c_Ac_Be_B + c_A^2 e_A + e_A^2c_B)  > 0  &  \  \Leftrightarrow \ c_Be_A^3>c_A c_B e_A e_B - e_A^2 c_A^2. 
\end{align*}
Accordingly, $ c_A^3 e_B < c_A c_B e_A e_B - e_A^2 c_A^2 < c_Be_A^3.$
The second statement is immediate by reversing the direction of the above inequalities.
\end{proof}

The hypotheses in Propositions~\ref{prop:ind-R-to-P}--\ref{prop:index-RSP} can be simplified in view of the previous lemma. 
We obtain the following more specific results concerning the stability regions of the three sub-cycles,
illustrated in Figures~\ref{fig:stability} and \ref{fig:stabilityAll}.

\begin{proposition}\label{prop:stab_as}
Let \eqref{hyp:as} be satisfied. Then:
\begin{enumerate}
\renewcommand{\theenumi}{\alph{enumi}}
\renewcommand{\labelenumi}{({\theenumi})}
	\item
	 the sub-cycle $\Sigma_\text{R-to-P}$ is
	\begin{enumerate}
	\renewcommand{\theenumii}{(\theenumi\arabic{enumii})}
	\renewcommand{\labelenumii}{{\theenumii}}
		\item  c.u.\ if $c_Ae_A < c_Be_B$;
		\item f.a.s.\ if $c_Ae_A > c_Be_B$;
		\item e.a.s.\ if $c_Ae_A - c_Be_B> e_A e_B$.
	\end{enumerate}
	\item\label{stab_Star} the cycle $\Sigma_\text{Star}$ is
	\begin{enumerate}
	\renewcommand{\theenumii}{(\theenumi\arabic{enumii})}
	\renewcommand{\labelenumii}{{\theenumii}}
		\item  c.u.\ if either $c_Be_B<c_Ae_A$ or $c_A^3e_B < c_B e_A^3$;
		\item  f.a.s.\ if $c_Be_B>c_Ae_A$ and $c_A^3e_B > c_B e_A^3$;.
	\end{enumerate}
	\item\label{RSPfas}  the sub-cycle $\Sigma_\text{RSP}$ is
	\begin{enumerate}
	\renewcommand{\theenumii}{(\theenumi\arabic{enumii})}
	\renewcommand{\labelenumii}{{\theenumii}}
		\item  c.u. if either $\alpha_T<0$, or $\theta_T<0$, or $\mu_T<0$, or $\nu_T<0$;
		\item  f.a.s.\ if $\theta_T>0$ and $\nu_T>0$.
	\end{enumerate}
\end{enumerate}
\end{proposition}

\begin{proof} This proof is deferred to Appendix~\ref{subsec:stab_cycles_as_network}.

We use Theorem 3.1 in Lohse \cite{Lohse2015} showing that if all stability indices are positive then the cycle is e.a.s. whereas f.a.s. is obtained from Lemma 2.5 in \cite{GarCas2019}. It follows from its definition in \cite{PodviginaAshwin2011} that a stability index equal to $-\infty$ implies the complete instability of the cycle. 

For the sub-cycle $\Sigma_\text{R-to-P}$ the proof consists in checking which entries in the transition matrices may be negative under the constraint \eqref{hyp:as} and using this information to obtain the stability indices. 
For the other two sub-cycles this can be done directly from Propositions~\ref{prop:ind-Star} and  \ref{prop:index-RSP}.
\end{proof}

\begin{remark}
The conditions for the f.a.s.\ of the sub-cycle $\Sigma_\text{RSP}$ can be written as a function of the eigenvalues as
$$
\frac{c_A^2e_B +e_A^2c_A}{e_A} < c_Be_B < \frac{c_A^2e_A+ e_A^2c_B}{c_A}.
$$
\end{remark}

Hypothesis \eqref{hyp:as} does not provide a complete description of the stability of the Four-node sub-cycle in parameter space. This can, of course, be obtained if values are assigned to all eigenvalues. The next result lists the most general results.

\begin{lemma}\label{lem:4node_as}
Let \eqref{hyp:as} be satisfied.
\begin{enumerate}
\renewcommand{\theenumi}{\alph{enumi}}
\renewcommand{\labelenumi}{({\theenumi})}
	\item  The  sub-cycle $\Sigma_\text{4-node}$ is c.u. if at least one of the following holds
\begin{enumerate}
\renewcommand{\theenumii}{(\theenumi\arabic{enumii})}
	\renewcommand{\labelenumii}{{\theenumii}}
	\item  $c_A>c_B$;
	
	\item  $c_A^3e_B < c_Be_A^3$.
\end{enumerate}
	
	\item  If the sub-cycle $\Sigma_\text{4-node}$ is f.a.s. then so is the cycle $\Sigma_\text{Star}$.
\end{enumerate}
\end{lemma}

\begin{proof}
Let $c_A> c_B$ and write $\nu_T$, from Proposition~\ref{prop:index-4node}, as in~\eqref{eq:bn-T}.
We have $\nu_T > 0$ since $c_A e_A - c_B e_B >0$.

Consider now two necessary conditions for $\Sigma_\text{4-node}$  not to be c.u:
$\theta_T<0$ and $\nu_T<0$. Given (c) in Lemma~\ref{lem:relations_as}, we find that 
$c_A^3e_B < c_Be_A^3$ makes the conditions incompatible.

Assuming that the Four-node cycle is f.a.s. we must have $\theta_T<0$ and $\nu_T<0$. Again from (c) in Lemma~\ref{lem:relations_as}, this implies $c_A^3e_B>c_Be_A^3$. From \eqref{eq:bn-T}, it is easy to see that $c_Be_B>c_Ae_A$ follows from $\nu_T<0$.
\end{proof}

\begin{figure}
\begin{center}
\parbox{45mm}{
\begin{center}
\includegraphics[width=45mm]{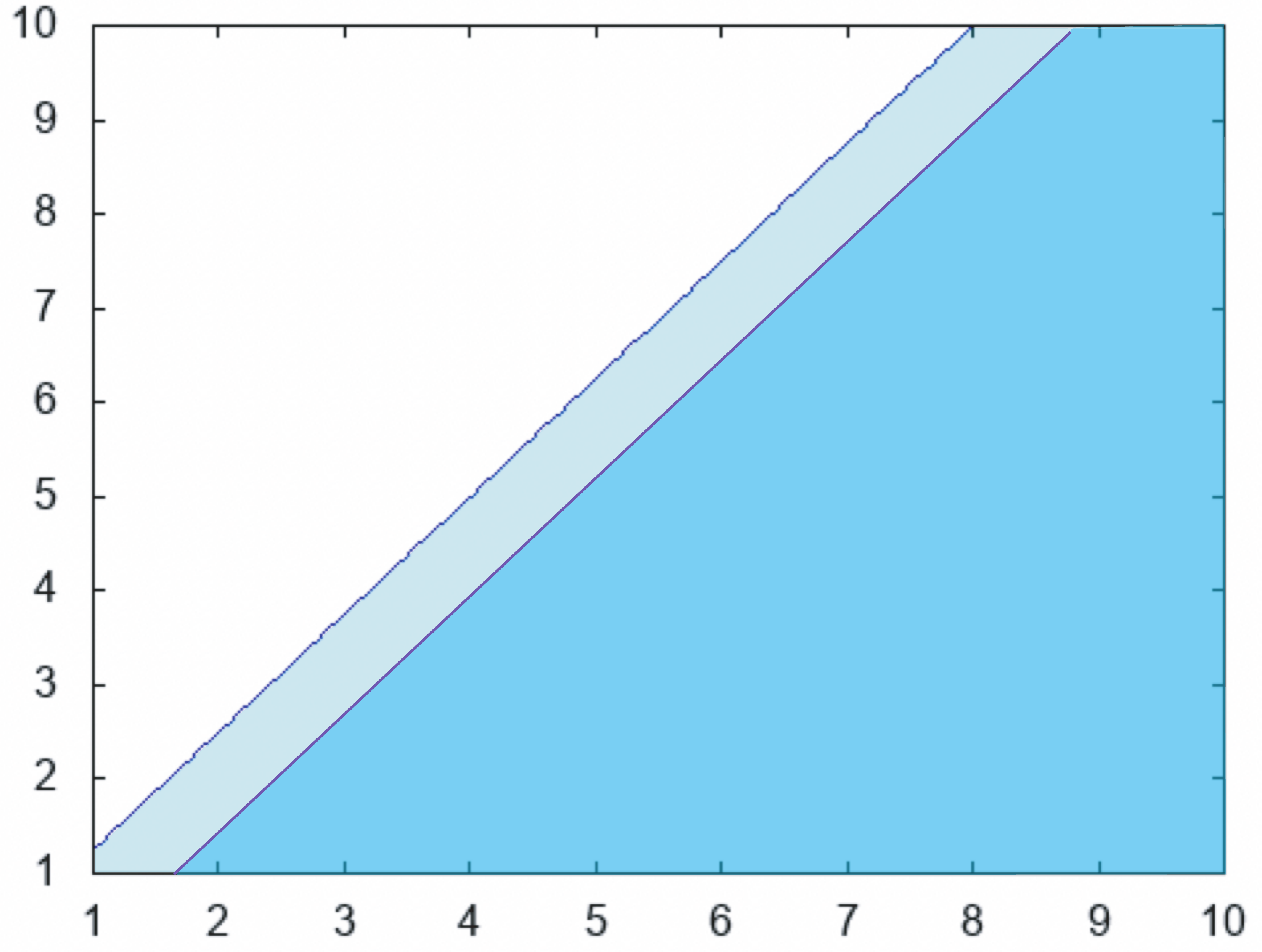}\\
R-to-P
\end{center}
}
\quad
\parbox{45mm}{
\begin{center}
\includegraphics[width=45mm]{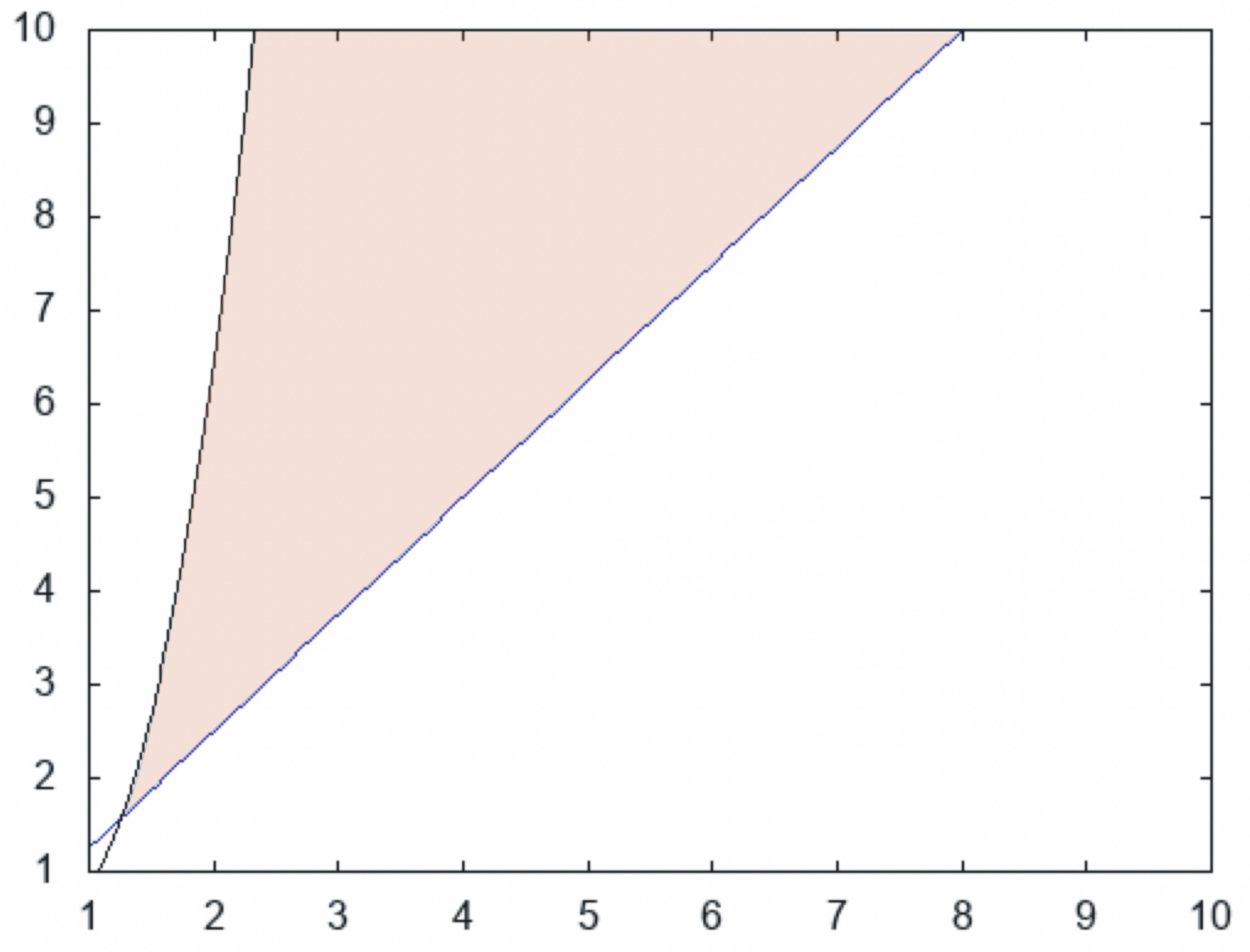}\\
Star
\end{center}
}
\quad
\parbox{45mm}{
\begin{center}
\includegraphics[width=45mm]{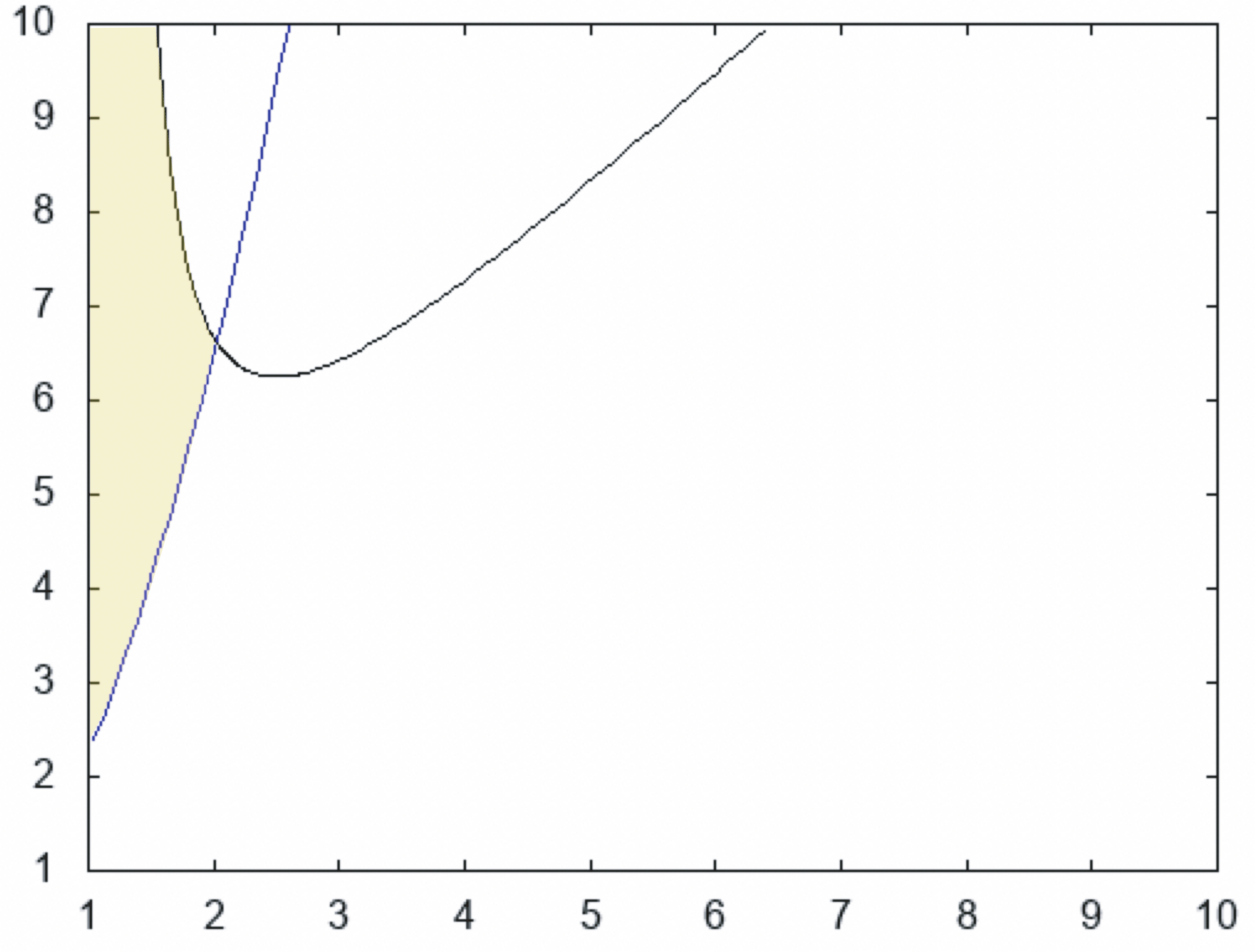}\\
RSP
\end{center}
}
\end{center}
\caption{Regions of stability in the $(c_A,c_B)$ plane for the three elementary cycles, with $e_A=1$, $e_B=0.8$,
in the region where the network is asymptotically stable. In the coloured region the (sub-)cycles  R-to-P,  Star and RSP  are f.a.s. and the sub-cycle R-to-P  is e.a.s. in the darker region. The Four-node sub-cycle is c.u. for at least those parameter values for which the Star cycle is c.u.}\label{fig:stability}
\end{figure}

\begin{figure}
\begin{center}
\includegraphics[width=60mm]{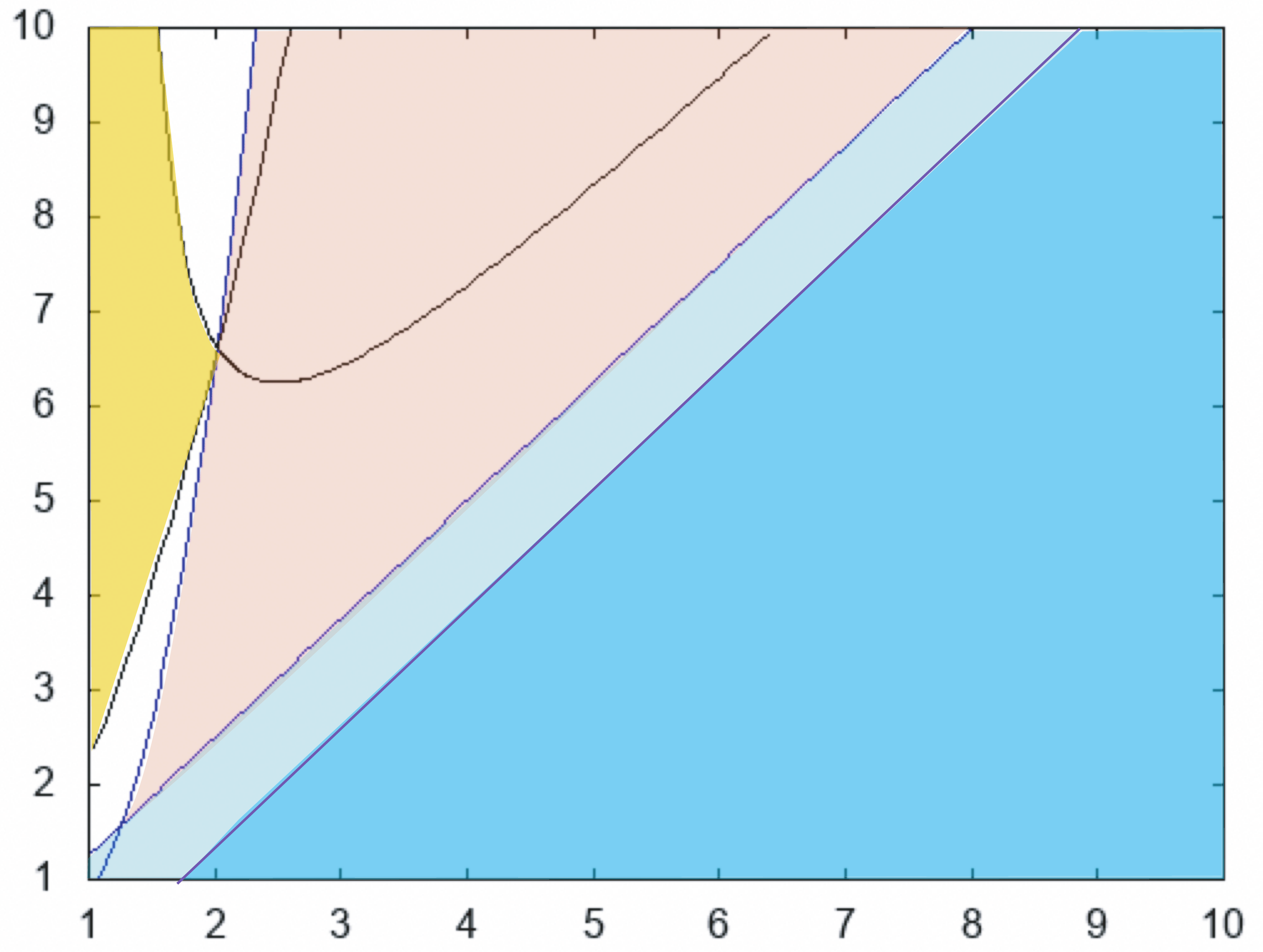}
\end{center}
\caption{Regions of stability in the $(c_A,c_B)$ plane for the three elementary cycles, with $e_A=1$, $e_B=0.8$,
in the region where the network is asymptotically stable.
The sub-cycle RSP is  f.a.s.\  at the yellow region on the left, Star is  f.a.s.\  on the pink region at the centre and 
R-to-P  is f.a.s. on the light blue region and e.a.s. on the darker region at the right. The Four-node sub-cycle is c.u. for at least those parameter values for which the Star cycle is c.u.
In the two white regions all the elementary cycles are c.u.\  The sausages found in \cite{PosRuc} lie in the lower component of the white region. }\label{fig:stabilityAll}
\end{figure}

A straightforward consequence of Lemma~\ref{lem:4node_as} is that the region of stability (f.a.s.) of the Four-node sub-cycle is contained in that of the Star cycle.

Our final result establishes some stability combinations for the cycles in the RSPLS network. The fact that all four 
(sub-)cycles may be c.u. indicates that other sequences may be visible in simulations.

\begin{proposition}\label{prop:stab-relations}
Consider the three elementary (sub-)cycles with an odd number of nodes.
 If \eqref{hyp:as} holds and one of these
 three elementary (sub-)cycles 
 satisfies the conditions above to be f.a.s. then the other two  
 elementary (sub-)cycles are c.u. 
Moreover, if $\nu_T<0$ and $c_A^3e_B < c_Be_A^3$, then all four 
(sub-)cycles are c.u.
\end{proposition}

\begin{proof}
According to Proposition~\ref{prop:stab_as},
$\Sigma_\text{Star}$ and $\Sigma_\text{RSP}$
are at most f.a.s. If $\Sigma_\text{R-to-P}$ is not c.u. then it is 
automatically f.a.s.

Suppose that $\Sigma_\text{R-to-P}$ is f.a.s. From (a1) in Proposition~\ref{prop:stab_as} we get $c_Ae_A  > c_Be_B$. Lemma~\ref{lem:no2stable} leads to $\Sigma_\text{Star}$ being c.u.  
The contrapositive of (b) in Lemma~\ref{lem:relations_as} determines that $c_Ae_A  > c_Be_B$ yields $\alpha_T, \theta_T <0$ and $\nu_T>0$. 
Proposition~\ref{prop:index-RSP} states that the stability indices for $\Sigma_\text{RSP}$ are all equal to $-\infty$ and 
this sub-cycle is also c.u.

Suppose that $\Sigma_\text{Star}$ is f.a.s. Recall that Lemma~\ref{lem:no2stable} already establishes that  $\Sigma_\text{R-to-P}$ is c.u.
From the contrapositive of (c) in Lemma~\ref{lem:relations_as}, if $c_A^3e_B>c_Be_A^3$ then
either $\theta_T<0$ or $\nu_T<0$, and thus $\Sigma_\text{RSP}$ is also c.u.

Suppose that $\Sigma_\text{RSP}$ is f.a.s. 
By virtue of (b) in Lemma~\ref{lem:relations_as}, when $\theta_T > 0$ we have $c_A e_A < c_B e_B$.
That $\Sigma_\text{R-to-P}$ is c.u. is a consequence of (a1) in Proposition~\ref{prop:stab_as}.
Given (c) in Lemma~\ref{lem:relations_as}, it follows that $c_A^3e_B<c_Be_A^3$, and hence $\Sigma_\text{Star}$ is c.u.

To prove the second statement we use (b) in Lemma~\ref{lem:relations_as} whence
$\nu_T<0$ implies $c_Ae_A<c_Be_B$, ensuring that $\Sigma_\text{RSP}$ and $\Sigma_\text{R-to-P}$ are both c.u. The remaining condition asserts that $\Sigma_\text{Star}$ 
and $\Sigma_\text{4-node}$ are also c.u. 
\end{proof}

We finish this section by considering the parameter range depicted in Figure~7 of \cite{PosRuc} and add the information provided by our analysis. This allows us to distinguish f.a.s. from e.a.s. when a cycle is f.a.s. but not e.a.s.

Since the Star cycle is not e.a.s. for the chosen values of $e_A$ and $e_B$, and the stability indices are the same along all its connections, our results coincide in determining the stability region of this cycle. However, for the remaining two sub-cycles, we can add that there are smaller regions inside those identified in Figure~7 of \cite{PosRuc} where stronger attraction properties occur. 

Propositions~\ref{prop:ind-R-to-P},~\ref{prop:ind-Star} and  \ref{prop:index-RSP} also provide proof that the (sub-)cycles are f.a.s, but not e.a.s, in the regions depicted for their stability in \cite{PosRuc}. This is achieved by replacing $e_A=1$ and $e_B=0.8$ in the expressions for the values of the function $F^\text{index}$ in Propositions~\ref{prop:ind-R-to-P} and~\ref{prop:ind-Star} and by replacing these values directly into the stability indices calculated in Proposition~\ref{prop:index-RSP}. 

When the RSPLS network is a.s., the Rock-to-Paper sub-cycle is e.a.s.\ under condition~(a3) in Proposition~\ref{prop:stab_as} and, if we allow $c_A$ or $c_B$ to be smaller than $e_A$, then a lower bound appears to guarantee that the last of the values of $F^\text{index}$ in the proof of Proposition~\ref{prop:ind-R-to-P} is positive. Although the RSP sub-cycle is not e.a.s., it may be f.a.s. with either just one or two connections with a positive stability index ($\sigma_{23}<0$ always). The values obtained by replacing $e_A=1$ and $e_B=0.8$ in Proposition~\ref{prop:index-RSP} show that $\sigma_{12}>0$ and
\[
\sigma_{31}=\begin{cases}
1-0.8 c_B + c_A \ (<0) & \text{ if } \dfrac{1+c_A}{c_B} \leq 0.8 \\
\min\left\{ 0.25, -1 +\dfrac{1}{0.8c_B-c_A} \ (>0) \right\} & \text{ if } 0.8 < \dfrac{1+c_A}{c_B}.
\end{cases}
\]
Hence, the existence of at least one positive index may support the visibility of the RSP sub-cycle in simulations. In the region of stability but closer to its lower bound there are two connections along which the stability index is positive. This promotes the attraction properties of the cycle.

Note that the parameter region depicted in our figures is much wider than that analysed in \cite{PosRuc}. In fact, our results are analytic and therefore extend to values, not previously considered, of all four eigenvalues: $e_A$, $e_B$, $c_A$ and $c_B$.

\section{Concluding remarks}

We present a new and thorough analysis of the stability for the heteroclinic network describing the RSPLS game. We provide a detailed study of the stability of some sub-cycles while providing information for the interested reader to calculate the stability indices along any trajectory of the network. Our stability results support the findings of Postlethwaite and Rucklidge \cite{PosRuc} as well as some other simulations by other authors. 
At the same time, we establish stability results for a parameter range much wider than that in \cite{PosRuc}. Our results lead to the conjecture that interesting dynamics may be found in the white region at the top of our Figure~\ref{fig:stability}. 

In Vukov {\em et al.} \cite{VukSzoSza}, the variation of the invasion rates shows that ``two of the five species can become extinct within a short transient time and the system evolves into one of the three-species solutions''. This is consistent with the complete instability of connections of type A in the cycles of 5 nodes, while some weak stability is preserved for the RSP cycle. 
Such a result occurs within the region of stability for the sequence AAB in \cite[Figure~1]{PosRuc}.

The transitions among the five equivalent configurations of the RSP game appear in Cheng {\em et al.} \cite{Che_et_al2014} depicted in the Spiral Interaction Graph in their Figure~3, corresponding to the snapshots of their Figure~2(a).

Park {\em et al.} \cite{Par_et_al2017} report on the coexistence of all five species for certain parameter ranges describing the strength of competition, while
for some other parameter ranges less common coexistence patterns are observed. The coexistence of all five species corresponds to the stability of either of the heteroclinic cycles with 5 nodes. The other patterns reported in \cite{Par_et_al2017} include coexistence of 4 of the 5 species as a subnetwork without one node and coexistence of 3 of the 5 species in a non-cyclic manner. Both cases can be modelled by allowing for a more generic setting than that of \cite{PosRuc}, for example, by using the generic values of \cite{AfrMosYou} and tweaking the parameters to obtain the desired stability index for each connection in the network.

\paragraph{Acknowledgements:}
The first author is grateful to C.\ Postlethwaite for some useful clarifications of her work. The authors thank A.\ Lohse and the reviewers for some insightful remarks.

All authors were partially supported by Centro de Matem\'atica da Universidade do Porto (CMUP), financed by national funds through FCT - Funda\c{c}\~ao para a Ci\^encia e a Tecnologia, I.P., under the project UIDB/00144/2020. 
The third author is the recipient of the PhD grant number PD/BD/150534/2019 awarded by FCT - Funda\c{c}\~ao para a Ci\^encia e a Tecnologia which is co-financed by the Portuguese state budget, through the Ministry for Science, Technology and Higher Education (MCTES) and by the European Social Fund (FSE), through Programa Operacional Regional do Norte.

\appendix

\section{Transition matrices} 

\subsection{The Rock-to-Paper sub-cycle}  \label{app:R-to-P}

The products of basic transition matrices with respect to the Rock-to-Paper sub-cycle near~$\xi_2$ are: 
\[
\M_{(j+1,2)} = (\M_2)^{j} : \ H_2^{\inn} \rightarrow H_{j+2}^{\inn}, \quad j= 1,\ldots,5  \ (\text{mod } 5),
\]
where
\begin{align*}
\M_{(2,2)}  = \M_2 & = \begin{bmatrix}
\dfrac{c_B}{e_A} & 0 & 1 \\[0.4cm]
\dfrac{c_A}{e_A} & 0 & 0 \\[0.4cm]
-\dfrac{e_B}{e_A} & 1 & 0
\end{bmatrix}, 
\quad \M_{(3,2)}  = (\M_2)^2  = \begin{bmatrix}
\dfrac{c_B^2}{e_A^2} -\dfrac{e_B}{e_A} & 1 & \dfrac{c_B}{e_A} \\[0.4cm]
\dfrac{c_A c_B}{e_A^2}& 0 &  \dfrac{c_A}{e_A} \\[0.4cm]
-\dfrac{e_B c_B}{e_A^2} + \dfrac{c_A}{e_A}& 0 &  -\dfrac{e_B}{e_A}
\end{bmatrix},  \\[0.4cm]
\M_{(4,2)}  = (\M_2)^3 & = \begin{bmatrix}
\dfrac{c_B^3}{e_A^3} - \dfrac{2  \, e_B c_B}{e_A^2} + \dfrac{c_A}{e_A} & \dfrac{c_B}{e_A} & \dfrac{c_B^2}{e_A^2} - \dfrac{e_B}{e_A} \\[0.4cm]
\dfrac{c_B^2 c_A}{e_A^3} - \dfrac{c_A e_B}{e_A^2} & \dfrac{c_A}{e_A} & \dfrac{c_A c_B}{e_A^2} \\[0.4cm]
-\dfrac{c_B^2 e_B}{e_A^3} + \dfrac{c_A c_B + e_B^2}{e_A^2} & -\dfrac{e_B}{e_A}  & -\dfrac{e_B c_B}{e_A^2} + \dfrac{c_A}{e_A} 
\end{bmatrix},  \\[0.4cm]
\M_{(5,2)} = (\M_2)^4 & = \begin{bmatrix}
\dfrac{c_B^4}{e_A^4} - \dfrac{3 \, c_B^2 e_B}{e_A^3} + \dfrac{2 \, c_A c_B + e_B^2}{e_A^2}  & \dfrac{c_B^2}{e_A^2} - \dfrac{e_B}{e_A} & \dfrac{c_B^3}{e_A^3} - \dfrac{2 \, e_B c_B}{e_A^2} + \dfrac{c_A}{e_A} \\[0.4cm]
\dfrac{c_B^3 c_A}{e_A^4} - \dfrac{2 \, c_A c_B e_B}{e_A^3} + \dfrac{c_A^2}{e_A^2} & \dfrac{c_A c_B}{e_A^2} & \dfrac{c_B^2 c_A}{e_A^3} - \dfrac{c_A e_B}{e_A^2} \\[0.4cm]
-\dfrac{c_B^3 e_B}{e_A^4} + \dfrac{c_B^2 c_A + 2 \, e_B^2 c_B}{e_A^3}  - \dfrac{2 \, c_A e_B}{e_A^2} & -\dfrac{e_B c_B}{e_A^2} + \dfrac{c_A}{e_A}  & -\dfrac{c_B^2 e_B}{e_A^3} + \dfrac{c_A c_B + e_B^2}{e_A^2}
\end{bmatrix},
\end{align*}%
and
 \begin{multline*}
 \M_{(1,2)} = (\M_2)^5  = \left[ \begin{matrix}
\dfrac{c_B^5}{e_A^5} - \dfrac{4 \, c_B^3 e_B}{e_A^4} + \dfrac{3\, c_B^2 c_A+ 3 \, e_B^2 c_B}{e_A^3}  - \dfrac{2 \, c_A e_B}{e_A^2} \\[0.4cm]
 \dfrac{c_B^4 c_A}{e_A^5} - \dfrac{3 \, c_B^2 c_A e_B}{e_A^4} + \dfrac{2\, c_A^2 c_B + e_B^2 c_A}{e_A^3}  & \\[0.4cm]
-\dfrac{c_B^4 e_B}{e_A^5} + \dfrac{c_B^3 c_A + 3 \, c_B^2 e_B^2}{e_A^4} - \dfrac{4\, c_A c_B e_B + e_B^3}{e_A^3} +\dfrac{c_A^2}{e_A^2}   \end{matrix} \right. \\[0.4cm]
\left. \begin{matrix}
  \dfrac{c_B^3}{e_A^3} - \dfrac{2 \, e_B c_B}{e_A^2} + \dfrac{c_A}{e_A}  & \dfrac{c_B^4}{e_A^4} - \dfrac{3 \, c_B^2 e_B}{e_A^3} + \dfrac{2\, c_A c_B + e_B^2}{e_A^2}  \\[0.4cm]
 \dfrac{c_B^2 c_A}{e_A^3} - \dfrac{c_A e_B}{e_A^2}  & \dfrac{c_B^3 c_A}{e_A^4} - \dfrac{2\, c_A c_B e_B}{e_A^3} + \dfrac{c_A^2}{e_A^2}  \\[0.4cm]
   -\dfrac{c_B^2 c_A}{e_A^3} + \dfrac{c_A c_B + e_B^2}{e_A^2} & -\dfrac{c_B^3 e_B}{e_A^4} + \dfrac{c_B^2 c_A + 2\, e_B^2 c_B}{e_A^3} - \dfrac{2\, c_A e_B}{e_A^2}
\end{matrix} \right].
\end{multline*}
 
\subsection{The Star cycle} \label{app:Star}

The products of basic transition matrices with respect to the Star cycle near~$\xi_4$ are: 
\[
\M_{(3j+1,4)} = (\M_4)^j : \ H_4^{\inn} \rightarrow H_{3j+4}^{\inn},
\]
where
\begin{align*}
\M_{(4,4)}  = \M_4 & = \begin{bmatrix}
0 & \dfrac{c_A}{e_B} & 1 \\[0.4cm]
1 & -\dfrac{e_A}{e_B} & 0 \\[0.4cm]
0 & \dfrac{c_B}{e_B} & 0
\end{bmatrix}, 
\quad \M_{(2,4)}  = (\M_4)^2  = \begin{bmatrix}
\dfrac{c_A}{e_B} & -\dfrac{c_A e_A}{e_B^2} + \dfrac{c_B}{e_B} & 0 \\[0.4cm]
-\dfrac{e_A}{e_B}&  \dfrac{e_A^2}{e_B^2} +  \dfrac{c_A}{e_B} & 1 \\[0.4cm]
\dfrac{c_B}{e_B} &  -\dfrac{c_B e_A}{e_B^2} & 0
\end{bmatrix},  \\[0.4cm]
\quad \M_{(5,4)}  = (\M_4)^3 & = \begin{bmatrix}
-\dfrac{c_A e_A}{e_B^2} + \dfrac{c_B}{e_B} & \dfrac{e_A^2 c_A}{e_B^3} + \dfrac{c_A^2}{e_B^2} - \dfrac{c_B e_A}{e_B^2} & \dfrac{c_A}{e_B} \\[0.4cm]
\dfrac{e_A^2}{e_B^2} + \dfrac{c_A}{e_B} & -\dfrac{e_A^3}{e_B^3} - \dfrac{2\, c_A e_A}{e_B^2} + \dfrac{c_B}{e_B} & -\dfrac{e_A}{e_B} \\[0.4cm]
-\dfrac{c_B e_A}{e_B^2} & \dfrac{e_A^2 c_B}{e_B^3} + \dfrac{c_B c_A}{e_B^2} & \dfrac{c_B}{e_B} 
\end{bmatrix},  \\[0.4cm]
\M_{(3,4)}  = (\M_4)^4 & = \begin{bmatrix}
\dfrac{e_A^2 c_A}{e_B^3} + \dfrac{c_A^2- c_B e_A}{e_B^2} & -\dfrac{e_A^3 c_A}{e_B^4} + \dfrac{e_A^2 c_B - 2\, c_A^2 e_A}{e_B^3} + \dfrac{2\, c_B c_A}{e_B^2} & -\dfrac{c_A e_A}{e_B^2} +\dfrac{c_B}{e_B} \\[0.4cm]
-\dfrac{e_A^3}{e_B^3} - \dfrac{2\, c_A e_A}{e_B^2} +\dfrac{c_B}{e_B} & \dfrac{e_A^4}{e_B^4} + \dfrac{3 \, e_A^2 c_A}{e_B^3} + \dfrac{c_A^2 -2 \, c_B e_A}{e_B^2} & \dfrac{e_A^2}{e_B^2} + \dfrac{c_A}{e_B}  \\[0.4cm]
\dfrac{e_A^2 c_B}{e_B^3} + \dfrac{c_B c_A}{e_B^2} & -\dfrac{e_A^3 c_B}{e_B^4} - \dfrac{2 \, c_B e_A c_A}{e_B^3} + \dfrac{c_B^2}{e_B^2} & -\dfrac{c_B e_A}{e_B^2}
\end{bmatrix},
\end{align*}
and
\begin{multline*}
\M_{(1,4)}  = (\M_4)^5  = \left[ \begin{matrix}
-\dfrac{e_A^3 c_A}{e_B^4} +\dfrac{e_A^2 c_B - 2\, c_A^2 e_A}{e_B^3} +\dfrac{ 2\, c_B c_A}{e_B^2}  \\[0.4cm]
\dfrac{e_A^4}{e_B^4} + \dfrac{3 \, e_A^2 c_A}{e_B^3} + \dfrac{c_A^2 - 2\, c_B e_A}{e_B^2} \\[0.4cm]
-\dfrac{e_A^3 c_B}{e_B^4} - \dfrac{2 \, c_B e_A c_A}{e_B^3} + \dfrac{c_B^2}{e_B^2} \end{matrix} \right. \\[0.4cm]
\left.\begin{matrix}
\dfrac{e_A^4 c_A}{e_B^5} + \dfrac{3 \, c_A^2 e_A^2 - e_A^3 c_B}{e_B^4} +\dfrac{c_A^3 - 4\, c_B e_A c_A}{e_B^3} +\dfrac{c_B^2}{e_B^2}  & \dfrac{e_A^2 c_A}{e_B^3} + \dfrac{c_A^2 - c_B e_A}{e_B^2} \\[0.4cm]
- \dfrac{e_A^5}{e_B^5} -\dfrac{4 \, e_A^3 c_A}{e_B^4} + \dfrac{3 \, e_A^2 c_B - 3 c_A^2 e_A}{e_B^3} +\dfrac{2 \, c_B c_A}{e_B^2} & -\dfrac{e_A^3}{e_B^3} - \dfrac{2 \, c_A e_A}{e_B^2} + \dfrac{c_B}{e_B} \\[0.4cm]
\dfrac{e_A^4 c_B}{e_B^5} +\dfrac{3 \, e_A^2 c_A c_B}{e_B^4} + \dfrac{c_A^2 c_B - 2\, c_B^2 e_A}{e_B^3} & \dfrac{e_A^2 c_B}{e_B^3} + \dfrac{c_B c_A}{e_B^2}
\end{matrix} \right].
\end{multline*}
    
\subsection{The RSP sub-cycle} \label{subsec:RSP}

The products of basic transition matrices with respect to the RSP sub-cycle are:
\[
\begin{aligned}
\M_{(j,j+1)} & =  \M_{j+1} \M_j   & : \ & H_j^{\inn} \rightarrow H_{j+2}^{\inn}, \\[0.2cm]
\M^{(j)} & = \M_{j+2} \M_{j+1} \M_j   & : \ & H_j^{\inn}  \rightarrow H_j^{\inn}, \quad j= 1,2,3  \text{ (mod } 3)
\end{aligned}
\]
where
\begin{align}
\M_{(2,1)}  = \M_2 \M_1 & = \begin{bmatrix}
\dfrac{c_B^2}{e_A^2} -\dfrac{e_B}{e_A} & 1 & 0 \\[0.5cm]
\dfrac{c_A c_B}{e_A^2}& 0 & 0 \\[0.5cm]
-\dfrac{e_B c_B}{e_A^2} + \dfrac{c_A}{e_A}& 0 & 1
\end{bmatrix},  & 
\M^{(1)}  = \M_3 \M_2 \M_1 & = \begin{bmatrix}
\delta_T & 0 & 0 \\[0.5cm]
\alpha_T & 1 & 0 \\[0.5cm]
\beta_T & 0 & 1
\end{bmatrix} \label{eq:trans-M1}\\[0.5cm]
\M_{(3,2)} = \M_3 \M_2 & = \begin{bmatrix}
\dfrac{c_A^2}{e_B e_A} & 0 & 0 \\[0.5cm]
\dfrac{c_B}{e_A} - \dfrac{c_A}{e_B} & 0 & 1 \\[0.5cm]
\dfrac{c_B c_A}{e_B e_A} - \dfrac{e_B}{e_A} & 1 & 0
\end{bmatrix},  &  
\M^{(2)}  = \M_1 \M_3 \M_2 & = \begin{bmatrix}
\delta_T & 0 & 0 \\[0.5cm]
\gamma_T & 1 & 0 \\[0.5cm]
\theta_T & 0 & 1
\end{bmatrix} \label{eq:trans-M2} \\[0.5cm]
\M_{(1,3)} = \M_1 \M_3 & = \begin{bmatrix}
0 & \dfrac{c_B c_A}{e_B e_A} & 0  \\[0.5cm]
0 & \dfrac{c_A^2}{e_B e_A} + \dfrac{c_B}{e_B} & 1 \\[0.5cm]
1 & -\dfrac{c_A}{e_A} - \dfrac{e_A}{e_B} & 0
\end{bmatrix}, &  
\M^{(3)}  = \M_2 \M_1 \M_3 & = \begin{bmatrix}
1 & \mu_T & 0  \\[0.5cm]
0 & \delta_T & 0 \\[0.5cm]
0 & \nu_T & 1  
\end{bmatrix} \label{eq:trans-M3}
\end{align}
and
\begin{align*}
\delta_T & = \frac{c_A^2 c_B}{e_A^2 e_B} &  
\gamma_T & = \frac{c_A^3}{e_A^2 e_B} + \frac{c_B c_A}{e_B e_A} - \frac{e_B}{e_A} \\[0.2cm]
\alpha_T & = \frac{c_B^2}{e_A^2} - \frac{c_A c_B}{e_B e_A}- \frac{e_B}{e_A}  &
\theta_T & = - \frac{c_A^2}{e_A^2} +  \frac{c_B}{e_A} - \frac{c_A}{e_B}
\\[0.2cm]
\beta_T & =  \frac{c_B^2 c_A}{e_A^2 e_B} - \frac{e_B c_B}{e_A^2} + \frac{c_A}{e_A}  &
\mu_T & = \frac{c_B^2 c_A}{e_A^2 e_B} - \frac{c_A}{e_A} - \frac{e_A}{e_B} \\[0.2cm]
 & & 
 \nu_T & =  -\frac{c_B c_A}{e_A^2} + \frac{c_A^2}{e_A e_B} + \frac{c_B}{e_B}.
\end{align*}

\subsection{The Four-node sub-cycle} \label{app:4node}

The product of the basic transition matrices with respect to the Four-node sub-cycle are:
\begin{itemize}
\item near $\xi_1$
\end{itemize}
\[
\widehat{\M}_{(2,1)}  = \widehat{\M}_2 \widehat{\M}_1  = \begin{bmatrix}
\dfrac{c_A^2}{e_B e_A} & 0 & \dfrac{c_A}{e_B} \\[0.4cm]
\dfrac{c_B}{e_A} - \dfrac{c_A}{e_B} & 0 & -\dfrac{e_A}{e_B} \\[0.4cm]
\dfrac{c_B c_A}{e_B e_A} -\dfrac{e_B}{e_A}  & 1 & \dfrac{c_B}{e_B}
\end{bmatrix},
\]

\[
\widehat{\M}_{(5,1)}  = \widehat{\M}_5 \widehat{\M}_2 \widehat{\M}_1   = \begin{bmatrix}
-\dfrac{c_A^2}{e_B^2} + \dfrac{2\, c_B c_A}{e_B e_A} - \dfrac{e_B}{e_A} & 1 & -\dfrac{c_A e_A}{e_B^2} + \dfrac{c_B}{e_B} \\[0.4cm]
\dfrac{c_A e_A}{e_B^2} + \dfrac{c_A^2}{e_B e_A}  - \dfrac{c_B}{e_B} & 0 & \dfrac{e_A^2}{e_B^2} +\dfrac{c_A}{e_B} \\[0.4cm]
\dfrac{c_B^2}{e_B e_A} - \dfrac{c_B c_A}{e_B^2}  & 0 & -\dfrac{c_B e_A}{e_B^2}
\end{bmatrix}, 
\]

\[
\widehat{\M}^{(1)} = \widehat{\M}_3 \widehat{\M}_5 \widehat{\M}_2 \widehat{\M}_1   = \begin{bmatrix}
\dfrac{c_A^2 e_A}{e_B^3} + \dfrac{c_A^3}{e_B^2 e_A} - \dfrac{2 \, c_B c_A}{e_B^2} + \dfrac{c_B^2}{e_B e_A} & 0 & \dfrac{e_A^2 c_A}{e_B^3} + \dfrac{c_A^2-c_B e_A}{e_B^2} \\[0.4cm]
-\dfrac{e_A^2 c_A}{e_B^3} - \dfrac{2 \, c_A^2 - c_B e_A }{e_B^2} + \dfrac{2 \, c_B c_A}{e_B e_A} - \dfrac{e_B}{e_A} & 1 & -\dfrac{e_A^3}{e_B^3} - \dfrac{2 \, c_A e_A}{e_B^2} +\dfrac{c_B}{e_B} \\[0.4cm]
\dfrac{c_A e_A c_B}{e_B^3}  + \dfrac{c_A^2 c_B}{e_B^2 e_A} - \dfrac{c_B^2}{e_B^2} & 0 & \dfrac{e_A^2 c_B}{e_B^3} + \dfrac{c_B c_A}{e_B^2}
\end{bmatrix};
\]

\begin{itemize}
\item near $\xi_2$
\end{itemize}

\[
\widehat{\M}_{(5,2)}  =\widehat{\M}_5 \widehat{\M}_2  = \begin{bmatrix}
\dfrac{c_A}{e_B} & -\dfrac{c_A e_A}{e_B^2} + \dfrac{c_B}{e_B} & 1 \\[0.4cm]
-\dfrac{e_A}{e_B} & \dfrac{e_A^2}{e_B^2} + \dfrac{c_A}{e_B} & 0 \\[0.4cm]
\dfrac{c_B}{e_B} & \dfrac{c_B e_A}{e_B^2} & 0 \end{bmatrix}, 
\]

\[
\widehat{\M}_{(3,2)}  = \widehat{\M}_3 \widehat{\M}_5 \widehat{\M}_2   = \begin{bmatrix}
-\dfrac{c_A e_A}{e_B^2} + \dfrac{c_A}{e_B} & \dfrac{e_A^2 c_A}{e_B^3} + \dfrac{c_A^2 - c_B e_A}{e_B^2} & 0 \\[0.4cm]
\dfrac{e_A^2}{e_B^2} + \dfrac{c_A}{e_B} & -\dfrac{e_A^3}{e_B^3} - \dfrac{2 \, c_A e_A}{e_B^2} + \dfrac{c_B}{e_B} & 1 \\[0.4cm]
-\dfrac{c_B e_A}{e_B^2} & \dfrac{c_B e_A^2}{e_B^3} + \dfrac{c_B c_A}{e_B^2} & 0
\end{bmatrix}, 
\]

\[
\widehat{\M}^{(2)} =  \widehat{\M}_1 \widehat{\M}_3 \widehat{\M}_5 \widehat{\M}_2   = \begin{bmatrix}
-\dfrac{c_B c_A}{e_B^2} + \dfrac{c_B^2}{e_B e_A} & \dfrac{e_A c_B c_A}{e_B^3} + \dfrac{c_A^2 c_B}{e_B^2 e_A} - \dfrac{c_B^2}{e_B^2} & 0 \\[0.4cm]
-\dfrac{c_A^2+c_B e_A}{e_B^2} + \dfrac{c_B c_A}{e_B e_A} & \dfrac{c_A^2 e_A + e_A^2 c_B}{e_B^3} + \dfrac{c_A^3}{e_B^2 e_A} &  0 \\[0.4cm]
\dfrac{e_A^2}{e_B^2} + \dfrac{2 \, c_A}{e_B} -\dfrac{c_B}{e_A} & -\dfrac{e_A^3}{e_B^3} - \dfrac{3 \, c_A e_A}{e_B^2} - \dfrac{c_A^2}{e_B e_A} + \dfrac{2 \, c_B}{e_B} & 1
\end{bmatrix};
\]

\begin{itemize}
\item near $\xi_5$
\end{itemize}

\[
\widehat{\M}_{(3,5)}  = \widehat{\M}_3 \widehat{\M}_5 = \begin{bmatrix}
\dfrac{c_A}{e_B} & -\dfrac{c_A e_A}{e_B^2} + \dfrac{c_B}{e_B} & 0 \\[0.4cm]
-\dfrac{e_A}{e_B} & \dfrac{e_A^2}{e_B^2} +  \dfrac{c_A}{e_B} & 1 \\[0.4cm]
\dfrac{c_B}{e_B} & -\dfrac{c_B e_A}{e_B^2} & 0  \end{bmatrix}, 
\]

\[
\widehat{\M}_{(1,5)}  = \widehat{\M}_1 \widehat{\M}_3 \widehat{\M}_5   = \begin{bmatrix}
\dfrac{c_B c_A}{e_B e_A} & - \dfrac{c_B c_A}{e_B^2} + \dfrac{c_B^2}{e_B e_A} & 0 \\[0.4cm]
\dfrac{c_A^2}{e_B e_A} + \dfrac{c_B}{e_B} & -\dfrac{c_A^2 + c_B e_A}{e_B^2} + \dfrac{c_B c_A}{e_B e_A} & 0 \\[0.4cm]
-\dfrac{c_A}{e_A} - \dfrac{e_A}{e_B} & \dfrac{e_A^2}{e_B^2} + \dfrac{2 \, c_A}{e_B} - \dfrac{c_B}{e_A} & 1
\end{bmatrix}, 
\]

\[
\widehat{\M}^{(5)} = \widehat{\M}_2  \widehat{\M}_1 \widehat{\M}_3 \widehat{\M}_5  = \begin{bmatrix}
\dfrac{c_A^3}{e_B^2 e_A} + \dfrac{c_B c_A}{e_B^2} & - \dfrac{c_A^3 + e_A c_B c_A}{e_B^3} + \dfrac{c_A^2 c_B}{e_B^2 e_A} & 0  \\[0.4cm]
-\dfrac{c_A^2 + c_B e _A}{e_B^2} + \dfrac{c_B c_A}{e_B e_A} & \dfrac{c_A^2 e_A + e_A^2 c_B}{e_B^3} - \dfrac{2 \, c_B c_A}{e_B^2} + \dfrac{c_B^2}{e_B e_A} & 0 \\[0.4cm]
\dfrac{c_A^2 c_B}{e_B^2 e_A} + \dfrac{c_B^2}{e_B^2} - \dfrac{c_A}{e_A} - \dfrac{e_A}{e_B} & - \dfrac{c_A^2 c_B + c_B^2 e_A}{e_B^3} + \dfrac{c_B^2 c_A}{e_B^2 e_A} + \dfrac{e_A^2}{e_B^2} + \dfrac{2 \, c_A}{e_B} -\dfrac{c_B}{e_A} & 1
\end{bmatrix};
\]

\begin{itemize}
\item near $\xi_3$
\end{itemize}

\[
\widehat{\M}_{(1,3)}  = \widehat{\M}_1 \widehat{\M}_3  = \begin{bmatrix}
0 & \dfrac{c_B c_A}{e_B e_A} & \dfrac{c_B}{e_A} \\[0.4cm]
0 & \dfrac{c_A^2}{e_B e_A} + \dfrac{c_B}{e_B} & \dfrac{c_A}{e_A} \\[0.4cm]
1 & -\dfrac{c_A}{e_A} - \dfrac{e_A}{e_B} & -\dfrac{e_B}{e_A} \end{bmatrix}, 
\]

\[
\widehat{\M}_{(2,3)}  = \widehat{\M}_2 \widehat{\M}_1 \widehat{\M}_3  = \begin{bmatrix}
0 & \dfrac{c_A^3}{e_B^2 e_A} + \dfrac{c_B c_A}{e_B^2} & \dfrac{c_A^2}{e_B e_A} \\[0.4cm]
0 & -\dfrac{c_A^2 + c_B e_A}{e_B^2} + \dfrac{c_B c_A}{e_B e_A}  & \dfrac{c_B}{e_A} - \dfrac{c_A}{e_B} \\[0.4cm]
1 & \dfrac{c_A^2 c_B}{e_B^2 e_A} + \dfrac{c_B^2}{e_B^2} - \dfrac{c_A}{e_A} - \dfrac{e_A}{e_B} & \dfrac{c_B c_A}{e_B e_A} - \dfrac{e_B}{e_A}
\end{bmatrix}, 
\]

\[
\widehat{\M}^{(3)} =  \widehat{\M}_5 \widehat{\M}_2  \widehat{\M}_1 \widehat{\M}_3  = \begin{bmatrix}
1 & -\dfrac{c_A^3 + e_A c_B c_A}{e_B^3} + \dfrac{2 \, c_A^2 c_B}{e_B^2 e_A} + \dfrac{c_B^2}{e_B^2} - \dfrac{c_A}{e_A} - \dfrac{e_A}{e_B} & - \dfrac{c_A^2}{e_B^2} + \dfrac{2 \, c_B c_A}{e_B e_A} - \dfrac{e_B}{e_A} \\[0.4cm]
 0 & \dfrac{c_A^2 e_A + e_A^2 c_B}{e_B^3} + \dfrac{c_A^3}{e_B^2 e_A} & \dfrac{c_A e_A}{e_B^2} + \dfrac{c_A^2}{e_B e_A} - \dfrac{c_B}{e_B} \\[0.4cm]
0 & -\dfrac{c_A^2 c_B + c_B^2 e_A}{e_B^3} + \dfrac{c_B^2 c_A}{e_B^2 e_A} & -\dfrac{c_B c_A}{e_B^2} + \dfrac{c_B^2}{e_B e_A}
\end{bmatrix}.
\]

\bigskip

\section{Proofs}

\subsection{Proposition~\ref{prop:index-RSP}} \label{subsec:proof}

Given the entries of the transition matrices $M^{(j)}$, $j=1,2,3$, in \eqref{eq:trans-M1}--\eqref{eq:trans-M3} we derive the relations below
\begin{align*}
 \frac{c_B}{e_B} (\delta_T -1) + \nu_T & = \frac{c_A}{e_B} \beta_T,  & \frac{c_B}{e_A}\theta_T + \frac{e_B}{e_A} (\delta_T-1) & = \alpha_T,  \\[0.2cm]
\frac{c_A}{e_A}(\delta_T-1) + \beta_T  & = \frac{c_B}{e_A} \gamma_T, &  \frac{c_A}{e_A} \nu_T + \frac{e_B}{e_A} (\delta_T-1) & = \gamma_T, \\[0.2cm]
\frac{c_B}{e_A}(\delta_T-1) + \theta_T  & = \frac{c_A}{e_A} \mu_T, &  \frac{c_A}{e_B} \alpha_T + \frac{e_A}{e_B} (\delta_T - 1) & = \mu_T,
\end{align*}
which enable one to formulate the following:  
\begin{lemma} \label{lem:entries}
Suppose that $\delta_T>1$.
\begin{enumerate} 
\renewcommand{\theenumi}{(\alph{enumi})}
\renewcommand{\labelenumi}{{\theenumi}}
\item
If $\theta_T>0$ and $\nu_T>0$, then $\alpha_T>0$, $\beta_T>0$, $\gamma_T>0$ and $\mu_T>0$. 
\item 
If $\gamma_T<0$ and $\mu_T<0$, then $\alpha_T<0$, $\beta_T<0$, $\theta_T<0$ and $\nu_T<0$.
\end{enumerate}
\end{lemma}

\bigskip

\paragraph{Proof of Proposition~\ref{prop:index-RSP}:}

We start with checking conditions \emph{(i)--(iii)} for each $\M^{(j)}$, $j=1,2,3$, in \eqref{eq:trans-M1}--\eqref{eq:trans-M3}. Due to similarity, the transition matrices $\M^{(j)}$ have all the same eigenvalues. By the fact that $\M^{(1)}$ is a lower triangular matrix, the eigenvalues are the entries in the main diagonal: $\lambda_1 = \delta_T$ and $\lambda_2= \lambda_3 = 1$. 
Condition~\emph{(i)} is naturally satisfied by taking $\lambda_{\max}=\delta_T$. For each $j=1,2,3$, denote by $\boldsymbol{w}^{\max,j}$ the eigenvector of $\M^{(j)}$ associated with the eigenvalue $\lambda_{\max}$. 
An easy computation shows that $\boldsymbol{w}^{\max,1}=(\delta_T-1, \, \alpha_T, \, \beta_T)^\text{T}$, $\boldsymbol{w}^{\max,2}=(\delta_T-1, \, \gamma_T, \, \theta_T)^\text{T}$ and $\boldsymbol{w}^{\max,3}=(\mu_T, \, \delta_T-1, \, \nu_T)^\text{T}$.

Condition~\emph{(ii)} is violated when $\delta_T<1$ while condition~\emph{(iii)} is violated for some $j$ when $\delta_T<1$ or $\alpha_T<0$ or $\beta_T<0$ or $\gamma_T<0$ or $\theta_T<0$ or $\mu_T<0$ or $\nu_T<0$. Proposition~\ref{prop:calc-index}(a) then establishes statement in (a).

On the other hand, that conditions~\emph{(ii)}--\emph{(iii)} hold true when $\delta_T>1$, $\theta_T > 0$ and $\nu_T > 0$ follows from Lemma~\ref{lem:entries}. 
Under these inequalities all $\M^{(j)}$ meet~\eqref{eq:U-infty} as a result of any $\boldsymbol{y} \in \R_-^3$ written in the eigenbasis of $\M^{(j)}$ having a negative coefficient for the largest eigenvector. Indeed, such a coefficient is of the form $\left(\boldsymbol{v}^{\max,j}\right)^\text{T} \boldsymbol{y}<0$ given that $\boldsymbol{v}^{\max,1}=\boldsymbol{v}^{\max,2}=\left(\frac{1}{\delta_T-1}, \, 0, \,0\right)^\text{T}$ and $\boldsymbol{v}^{\max,3}=\left(0, \, \frac{1}{\delta_T-1}, \, 0\right)^\text{T}$ admit all non-negative components.

In the calculation of $\sigma_j$ we evaluate the function $F^{\ind}$ for the rows of the transition matrices $\M_j$, $\M_{j+1} \M_j$ and $\M^{(j)}=\M_{j+2}\M_{j+1}\M_j$ $(j \text{ mod  } 3)$ so that
\begin{align*}
\sigma_{31} = \begin{aligned}[t]\min \Bigg\{ & F^{\ind}\left( \frac{c_B}{e_A},0,0\right), F^{\ind}\left( \frac{c_A}{e_A},0,1\right), F^{\ind}\left( -\frac{e_B}{e_A},1,0\right), \\[0.2cm]
& F^{\ind}\left( \frac{c_B^2}{e_A^2} - \frac{e_B}{e_A},1,0\right), F^{\ind}\left( \frac{c_A c_B}{e_A^2},0,0\right), F^{\ind}\left(- \frac{e_Bc_B}{e_A^2} + \frac{c_A}{e_A},1,0\right),  \\[0.2cm]
&F^{\ind}\left(\delta_T, 0, 0\right), F^{\ind}\left(\alpha_T, 1, 0\right),  F^{\ind}\left(\beta_T, 0, 1\right) \Bigg\},
\end{aligned}
\end{align*}

\begin{align*}
\sigma_{12} = \begin{aligned}[t]\min \Bigg\{ & F^{\ind}\left( \frac{c_B}{e_A},0,1\right), F^{\ind}\left( \frac{c_A}{e_A},0,0\right), F^{\ind}\left( -\frac{e_B}{e_A},1,0\right), \\[0.2cm]
& F^{\ind}\left( \frac{c_A^2}{e_B e_A},0,0\right), F^{\ind}\left( \frac{c_B}{e_A} - \frac{c_A}{e_B},0,1\right), F^{\ind}\left(\frac{c_Bc_A}{e_B e_A} - \frac{e_B}{e_A},1,0\right),  \\[0.2cm]
& F^{\ind}\left(\delta_T, 0, 0\right), F^{\ind}\left(\gamma_T, 1, 0\right),  F^{\ind}\left(\theta_T, 0, 1\right) \Bigg\},
\end{aligned}
\end{align*}
and 
\begin{align*}
\sigma_{23} = \begin{aligned}[t]\min \Bigg\{ & F^{\ind}\left( 0,\frac{c_A}{e_B},0\right), F^{\ind}\left( 1,-\frac{e_A}{e_B},0\right), F^{\ind}\left( 0,\frac{c_B}{e_B},1\right), \\[0.2cm]
& F^{\ind}\left( 0,\frac{c_B c_A}{e_B e_A},0\right), F^{\ind}\left( 0,\frac{c_A^2}{e_B e_A} + \frac{c_B}{e_B},0\right), F^{\ind}\left( 1,-\frac{c_A}{e_A} - \frac{e_A}{e_B},0\right), \\[0.2cm]
& F^{\ind}\left(1,\mu_T, 0\right), F^{\ind}\left(0,\delta_T, 0\right),  F^{\ind}\left(0,\nu_T, 1\right) \Bigg\}.
\end{aligned}
\end{align*}
It is immediate that
\begin{align*}
F^{\ind}\left( \frac{c_B}{e_A},0,0\right) = F^{\ind}\left( \frac{c_A}{e_A},0,1\right)  & = +\infty \\
F^{\ind}\left( \frac{c_B^2}{e_A^2} - \frac{e_B}{e_A},1,0\right) & = +\infty \\
F^{\ind}\left(\delta_T,0,0\right) = F^{\ind}\left(\alpha_T,1,0\right) = F^{\ind}\left(\beta_T,0,1\right) & = +\infty,
\end{align*}
\begin{align*}
F^{\ind}\left( \frac{c_B}{e_A},0,1\right) = F^{\ind}\left( \frac{c_A}{e_A},0,0\right) & = +\infty \\
F^{\ind}\left( \frac{c_A^2}{e_B e_A},0,0\right) & = +\infty \\
F^{\ind}\left(\delta_T, 0, 0\right) = F^{\ind}\left(\gamma_T, 1, 0\right) = F^{\ind}\left(\theta_T, 0, 1\right) & = +\infty,
\end{align*}
\begin{align*}
F^{\ind}\left( 0,\frac{c_A}{e_B},0\right) = F^{\ind}\left( 0,\frac{c_B}{e_B},1\right) & = +\infty  \\
F^{\ind}\left( 0,\frac{c_B c_A}{e_B e_A},0\right) = F^{\ind}\left( 0,\frac{c_A^2}{e_B e_A} + \frac{c_B}{e_B},0\right) & = +\infty  \\
F^{\ind}\left(1,\mu_T, 0\right) = F^{\ind}\left(0,\delta_T, 0\right) = F^{\ind}\left(0,\nu_T, 1\right)  & = +\infty.
\end{align*}
Moreover,
\[
\begin{aligned}
\alpha_T > 0 & \  \Leftrightarrow  \ \frac{c_B^2}{e_A^2} - \frac{e_B}{e_A} > \frac{c_A c_B}{e_B e_A} > 0    \\
\theta_T > 0 &  \ \Leftrightarrow \ \frac{c_B}{e_A} - \frac{c_A}{e_B} > \frac{c_A^2}{e_A^2} > 0  \  \Leftrightarrow \ -\frac{e_Bc_B}{e_A^2} + \frac{c_A}{e_A} = -\frac{e_B}{e_A} \left(  \frac{c_B}{e_A} - \frac{c_A}{e_B} \right) < 0
\end{aligned}
\]
and
\[
F^{\ind} \left( \frac{c_B^2}{e_A^2} - \frac{e_B}{e_A}, 1, 0 \right) = F^{\ind} \left(  \frac{c_B}{e_A} - \frac{c_A}{e_B}, 0, 1\right) = + \infty.
\]
It follows that
\begin{align*}
\sigma_{31} &  =  \min \left\{ F^{\ind}\left( -\frac{e_B}{e_A},1,0\right), F^{\ind}\left(- \frac{e_Bc_B}{e_A^2} + \frac{c_A}{e_A},1,0\right)\right\} \\
 \sigma_{12}  & = \min \left\{ F^{\ind}\left( -\frac{e_B}{e_A},1,0\right),  F^{\ind}\left(\frac{c_Bc_A}{e_B e_A} - \frac{e_B}{e_A},1,0\right)  \right\}  \\
\sigma_{23} & = \min \left\{ F^{\ind}\left( 1,-\frac{e_A}{e_B},0\right), F^{\ind}\left( 1,-\frac{c_A}{e_A} - \frac{e_A}{e_B},0\right) \right\}.
\end{align*}
We get
\begin{align*}
F^{\ind}\left( -\frac{e_B}{e_A},1,0\right) & =
\begin{cases}
1 - \dfrac{e_B}{e_A}  \; \;  (<0)  & \text{ if } \dfrac{e_B}{e_A} > 1 \\[0.5cm]
- 1 + \dfrac{e_A}{e_B} \; \;  (>0)  & \text{ if } \dfrac{e_B}{e_A} < 1
\end{cases} \\
F^{\ind}\left(- \frac{e_Bc_B}{e_A^2} + \frac{c_A}{e_A},1,0\right) & = 
\begin{cases}
1 - \dfrac{e_Bc_B}{e_A^2} + \dfrac{c_A}{e_A}  \; \;  (<0) & \text{ if }  - \dfrac{e_Bc_B}{e_A^2} + \dfrac{c_A}{e_A}  \leq - 1 \\[0.5cm] 
-1 + \dfrac{e_A^2}{e_B c_B - c_A e_A }\; \;  (>0) & \text{ if }  -1 < - \dfrac{e_Bc_B}{e_A^2} + \dfrac{c_A}{e_A}  < 0, 
\end{cases} \\[0.5cm]
F^{\ind}\left(\frac{c_Bc_A}{e_B e_A} - \frac{e_B}{e_A},1,0\right) & = 
\begin{cases}
+ \infty & \text{ if }  \dfrac{c_Bc_A}{e_B e_A} - \dfrac{e_B}{e_A} \geq 0 \\[0.5cm]
1+  \dfrac{c_Bc_A}{e_B e_A} - \dfrac{e_B}{e_A}  \; \;  (<0)  & \text{ if }  \dfrac{c_Bc_A}{e_B e_A} - \dfrac{e_B}{e_A}  \leq -1 \\[0.5cm] 
-1+\dfrac{e_A e_B }{e_B^2 - c_A c_B} \; \;  (>0) & \text{ if }  -1 <  \dfrac{c_Bc_A}{e_B e_A} - \dfrac{e_B}{e_A} < 0, \\ 
\end{cases} \\[0.5cm]
F^{\ind}\left(1, -\frac{e_A}{e_B},0\right) & =
\begin{cases}
1 - \dfrac{e_A}{e_B}  \; \;  (<0)  & \text{ if } \dfrac{e_B}{e_A} < 1 \\[0.5cm]
- 1 + \dfrac{e_B}{e_A} \; \;  (>0)  & \text{ if } \dfrac{e_B}{e_A} > 1
\end{cases} \\
F^{\ind}\left( 1,-\frac{c_A}{e_A} - \frac{e_A}{e_B},0\right) & =
\begin{cases}
1 -\dfrac{c_A}{e_A} - \dfrac{e_A}{e_B}   \; \;  (<0)  & \text{ if } \dfrac{c_A}{e_A} + \dfrac{e_A}{e_B} >  1 \\[0.5cm] 
-1 + \dfrac{e_A e_B}{e_A^2 + c_A e_B }  \; \;  (>0)  & \text{ if } \dfrac{c_A}{e_A} + \dfrac{e_A}{e_B} <  1.
\end{cases}
\end{align*}
By combining all suitable branches for each $\sigma_{ij}$, the proof is completed.

\subsection{Proposition~\ref{prop:index-4node}}\label{subsec:proof4nodes}

We determine the eigenvalues and eigenvectors of $\widehat{\M}^{(j)}$, $j=1,2,3,5$, in Appendix~\ref{app:4node}. To simplify consider the following notation:
\begin{align*}
\widehat{\M}^{(1)} & = \begin{bmatrix}
\alpha_{11} & 0 & \alpha_{13} \\[0.1cm]
\alpha_{21} & 1 & \alpha_{23} \\[0.1cm]
\alpha_{31} & 0 & \alpha_{33}
\end{bmatrix}, 
& \widehat{\M}^{(2)} & = \begin{bmatrix}
\beta_{11} & \beta_{12} & 0  \\[0.1cm]
\beta_{21} & \beta_{22} & 0 \\[0.1cm]
\beta_{31} & \beta_{32} & 1  
\end{bmatrix}, \\[0.4cm]
  \widehat{\M}^{(5)} & = \begin{bmatrix}
\gamma_{11} & \gamma_{12} & 0  \\[0.1cm]
\gamma_{21} & \gamma_{22} & 0 \\[0.1cm]
\gamma_{31} & \gamma_{32} & 1  
\end{bmatrix},
&  \widehat{\M}^{(3)} & = \begin{bmatrix}
1 & \delta_{12} & \delta_{13} \\[0.1cm]
0 & \delta_{22} & \delta_{23} \\[0.1cm]
0 & \delta_{32} & \delta_{33}
\end{bmatrix}.
\end{align*}
We observe that

\begin{align}
\alpha_{13} & = \frac{c_B}{e_A} \alpha_{31} = \frac{e_A}{c_B} \beta_{12} = \frac{e_A}{e_B} \delta_{23} = - \frac{e_A^2}{e_B^2} \theta_T  \label{eq:alpha13}\\[0.2cm]
\gamma_{12} & =\frac{c_A}{e_B} \gamma_{21}  = \frac{e_B}{c_A} \beta_{21} =\frac{c_A}{c_B} \delta_{32} = -\frac{c_A e_A}{e_B^2} \nu_T. \label{eq:gamma12}
\end{align}
As $\widehat{\M}^{(j)}$ are similar, all have the same eigenvalues. For $\widehat{\M}^{(1)}$ as defined above the eigenvalues are
\begin{align*}
\lambda_1 & =  \frac{\alpha_{11} + \alpha_{33} + \sqrt{\left(\alpha_{11} + \alpha_{33} \right)^2 - 4 \, \frac{c_B^3 c_A}{e_B^3 e_A}}}{2} = \frac{\alpha_{11} + \alpha_{33} + \sqrt{\left(\alpha_{11} - \alpha_{33} \right)^2 + 4 \alpha_{13}  \alpha_{31}}}{2}, \\
\lambda_2 & =  \frac{\alpha_{11} + \alpha_{33} - \sqrt{\left(\alpha_{11} + \alpha_{33} \right)^2 - 4 \, \frac{c_B^3 c_A}{e_B^3 e_A}}}{2} = \frac{\alpha_{11} + \alpha_{33} - \sqrt{\left(\alpha_{11} - \alpha_{33}\right)^2 + 4 \alpha_{13}  \alpha_{31}}}{2}, \\[0.2cm]
\lambda_3 & = 1.
\end{align*}
The candidate for $\lambda_{\max}$ satisfying conditions~\emph{(i)}--\emph{(iii)} is~$\lambda_1$. 
From~\eqref{eq:alpha13} 
we find that 
\begin{equation}
\label{eq:product-a}
\alpha_{13} \alpha_{31} = \frac{c_B}{e_A} \alpha_{13}^2 = 
\frac{e_A^4}{e_B^4} \theta_T^2  \geq 0,
\end{equation}
which ensures that $\lambda_{1}$ and $\lambda_2$ are real. Condition~\emph{(i)} is immediately true. Moreover,
\[
\alpha_{11} + \alpha_{33} =  \frac{c_A^3 e_B + \left(c_A e_A - c_B e_B \right)^2 + c_A c_B e_A e_B +  c_B e_A^3}{e_B^3 e_A} > 0
\]
and, in consequence, $\lambda_1 > |\lambda_2|$.

Let $\lambda_{\max}=\lambda_1$ and $\boldsymbol{w}^{\max,j} = \left(w_1^{\max,j}, w_2^{\max,j}, w_3^{\max,j}\right)^\text{T}$ the corresponding eigenvector of each $\widehat{\M}^{(j)}$. We get
\begin{align*}
\boldsymbol{w}^{\max,1} & = \Big(  \alpha_{13} \left(\lambda_{\max} - 1 \right), \ \alpha_{13} \alpha_{21} + \alpha_{23} \left(\lambda_{\max} - \alpha_{11} \right), \ \left(\lambda_{\max} - \alpha_{11} \right)  \left(\lambda_{\max} - 1 \right) \Big)^\text{T} \\
\boldsymbol{w}^{\max,2} & = \Big(  \beta_{12}  \left(\lambda_{\max} - 1 \right), \ \left(\lambda_{\max} - \beta_{11} \right)  \left(\lambda_{\max} - 1 \right), \ \beta_{12} \beta_{31} + \beta_{32} (\lambda_{\max} - \beta_{11}) \Big)^\text{T} \\
\boldsymbol{w}^{\max,5} & = \Big(  \gamma_{12}  \left(\lambda_{\max} - 1 \right), \ (\lambda_{\max} - \gamma_{11})  \left(\lambda_{\max} - 1 \right), \ \gamma_{12} \gamma_{31} + \gamma_{32} (\lambda_{\max} - \gamma_{11}) \Big)^\text{T} \\
\boldsymbol{w}^{\max,3} & = \Big( \delta_{12} \delta_{23} + \delta_{13} \left(\lambda_{\max} - \delta_{22} \right), \ \delta_{23}  \left(\lambda_{\max} - 1 \right), \ \left(\lambda_{\max} - \delta_{22} \right)  \left(\lambda_{\max} - 1 \right) \Big)^\text{T}.
\end{align*} 

Condition~\emph{(ii)} is satisfied if and only if
\[
\alpha_{11} + \alpha_{33} > \min \left\{ 2, 1 + \frac{c_B^3 c_A}{e_B^3 e_A}\right\}.
\]

Condition~\emph{(iii)} requires the evaluation of the signs of the components of $\boldsymbol{w}^{\max,j}$. 
Note that 
\begin{align*}
\left(\alpha_{11} - \alpha_{33}\right)^2 + 4 \alpha_{13}  \alpha_{31} &= \left(\beta_{11} - \beta_{22}\right)^2 + 4 \beta_{12}  \beta_{21} \\
& = \left(\gamma_{11} - \gamma_{22}\right)^2 + 4 \gamma_{12}  \gamma_{21} \\
& = \left(\delta_{22} - \delta_{33}\right)^2 + 4 \delta_{23}  \delta_{32}.
\end{align*}
It is easily seen that $\lambda_{\max} - \alpha_{11} > 0$ and $\lambda_{\max} - \gamma_{11} > 0$ 
because of \eqref{eq:product-a} and $\gamma_{12} \gamma_{21} = \frac{c_A}{e_B} \gamma_{21}^2  = 
\frac{c_A^2 e_A^2}{e_B^4} \nu_T^2  \geq 0$, 
respectively. Assuming $\lambda_{\max}>1$ to hold for every~$j$, we have $w_3^{\max,1} >0$ and $w_2^{\max,5} >0$. 
By~\eqref{eq:alpha13}
all components of $\boldsymbol{w}^{\max,j}$ have the same sign when
\begin{align*}
[j=1]& \quad \alpha_{13} > 0 \quad \text{and} \quad w_2^{\max,1}=\alpha_{13} \alpha_{21} + \alpha_{23} \left(\lambda_{\max} - \alpha_{11} \right) > 0, \\[0.2cm]
[j=2]& \quad  \lambda_{\max} - \beta_{11}>0  \quad \text{and} \quad  w_3^{\max,2}=\beta_{12} \beta_{31} + \beta_{32} (\lambda_{\max} - \beta_{11}) > 0, \\[0.2cm]
[j=5]& \quad  \gamma_{12} > 0 \quad \text{and} \quad  w_3^{\max,5}=\gamma_{12} \gamma_{31} + \gamma_{32} (\lambda_{\max} - \gamma_{11}) > 0, \\[0.2cm]
[j=3]& \quad  w_1^{\max,3}=\delta_{12} \delta_{23} + \delta_{13} \left(\lambda_{\max} - \delta_{22} \right) > 0 \quad \text{and} \quad \lambda_{\max} -\delta_{22} >0.
\end{align*}
Given \eqref{eq:gamma12} it means that $\alpha_{13} > 0$ and $\gamma_{12} > 0$ imply $\beta_{12} \beta_{21}  = \frac{c_A c_B}{e_A e_B}  \alpha_{13} \gamma_{12}> 0$ and $\delta_{23} \delta_{32} = \frac{c_A e_B}{c_B e_A} \alpha_{13} \gamma_{12} >0$. Hence, $\lambda_{\max} - \beta_{11} >0$ and $\lambda_{\max} - \delta_{22}>0$. 
What is left is to check 
\begin{equation}\label{eq:4node}
w_2^{\max,1}>0, \quad w_3^{\max,2}>0, \quad w_3^{\max,5}>0 \quad \text{ and } \quad w_1^{\max,3}>0.
\end{equation}
Using similarity, we can establish
\begin{align*}
\boldsymbol{w}^{\max,2}  &= \widehat{\M}_1 \boldsymbol{w}^{\max,1} \\[0.2cm]
\boldsymbol{w}^{\max,5}  &= \frac{e_B}{c_A} \frac{\gamma_{12}}{\lambda_{\max} - \beta_{11}}\widehat{\M}_2  \boldsymbol{w}^{\max,2} \\[0.2cm]
\boldsymbol{w}^{\max,3}  &= \frac{e_B}{c_B} \frac{\lambda_{\max}-\delta_{22}}{\lambda_{\max} - \gamma_{11}}\widehat{\M}_5  \boldsymbol{w}^{\max,5}.
\end{align*}
It follows that
\begin{align*}
w_3^{\max,2} & = -\frac{e_A}{e_B} \, \alpha_{13} \left(\lambda_{\max}-1\right) + w_2^{\max,1} \\[0.2cm]
w_3^{\max,5} & = \frac{e_B}{c_A} \frac{\gamma_{12}}{\lambda_{\max} - \beta_{11}} \left[ \frac{c_B}{e_B} \left(\lambda_{\max}-\beta_{11}\right) \left(\lambda_{\max}-1\right)+ w_3^{\max,2} \right]  \\[0.2cm]
 w_1^{\max,3} & = \frac{e_B}{c_B} \frac{\lambda_{\max}-\delta_{22}}{\lambda_{\max} - \gamma_{11}} \left[ \frac{c_A}{e_B} \left(\lambda_{\max}-\gamma_{11}\right) \left(\lambda_{\max}-1\right) + w_3^{\max,5} \right].
\end{align*}
Accordingly, if $w_3^{\max,2} >0$, then $w_2^{\max,1} >0$, $w_3^{\max,5} >0$ and $ w_1^{\max,3}>0$. We conclude that condition~\emph{(iii)} is fulfilled for
\begin{align*}
\alpha_{13} > 0 & \  \Leftrightarrow \ \theta_T < 0 \\[0.2cm]  
\gamma_{12} > 0 & \  \Leftrightarrow \  \nu_T < 0 \\[0.2cm] 
w_3^{\max,2} > 0 & \  \Leftrightarrow \  \begin{aligned}[t]  c_A^3 c_B & + 2\, c_B^2 c_A e_A - c_B^3 e_A e_B \\ & + \left( -e_A^4 - 3 \, e_A^2 c_A e_B-c_A^2 e_B^2 +2  \, e_B^2 c_B e_A\right)\lambda_{\max}>0. \end{aligned}
\end{align*}

The proof of (a) is immediate.
Under the hypotheses of (b), the equality~\eqref{eq:U-infty} holds for all $j$. In fact, any $\boldsymbol{y} \in \R_-^3$ written in the eigenbasis of $\widehat{\M}^{(j)}$ has a negative coefficient of the form $\left(\boldsymbol{v}^{\max,j}\right)^\text{T} \boldsymbol{y}$, where
\begin{align*}
\boldsymbol{v}^{\max,1} & = \left( \frac{\alpha_{11} - \lambda_2}{\left(\lambda_1-1\right) \left(\lambda_1-\lambda_2\right) \alpha_{13}}, \  0, \ \frac{1}{\left(\lambda_1-1\right) \left(\lambda_1-\lambda_2\right)}\right)^\text{T} \\
\boldsymbol{v}^{\max,2} & = \left( \frac{\beta_{11} - \lambda_2}{\left(\lambda_1-1\right) \left(\lambda_1-\lambda_2\right) \beta_{12}}, \ \frac{1}{\left(\lambda_1-1\right) \left(\lambda_1-\lambda_2\right)}, \ 0 \right)^\text{T} \\
\boldsymbol{v}^{\max,5} & = \left( \frac{\gamma_{11} - \lambda_2}{\left(\lambda_1-1\right) \left(\lambda_1-\lambda_2\right) \gamma_{12}}, \ \frac{1}{\left(\lambda_1-1\right) \left(\lambda_1-\lambda_2\right)}, \ 0 \right)^\text{T} \\
\boldsymbol{v}^{\max,3} & = \left( 0, \ \frac{\delta_{22} - \lambda_2}{\left(\lambda_1-1\right) \left(\lambda_1-\lambda_2\right) \delta_{13}}, \ \frac{1}{\left(\lambda_1-1\right) \left(\lambda_1-\lambda_2\right)}\right)^\text{T}.
\end{align*}
A trivial verification shows that $\lambda_1 - \lambda_2>0$, $\alpha_{11} - \lambda_2>0$, $\beta_{11} - \lambda_2>0$, $\gamma_{11} - \lambda_2>0$ and $\delta_{22} - \lambda_2>0$. Therefore, all $\boldsymbol{v}^{\max,j}$ admit non-negative components. 

We calculate the stability indices by plugging the rows of the transition matrices $\widehat{\M}_j$, $\widehat{\M}_{(l,j)}$, $\widehat{\M}^{(j)}$ with at least one negative entry into $F^\text{index}$. We thus get
\begin{align*}
\sigma_{31}  =   \min \bigg\{ & F^\textnormal{index} \left(-\frac{e_B}{e_A},1,0\right),  \\
& F^\textnormal{index}  \left( \frac{c_B}{e_A} - \frac{c_A}{e_B}, 0, -\frac{e_A}{e_B} \right), 
F^\textnormal{index}  \left( \frac{c_B c_A}{e_B e_A} - \frac{e_B}{e_A}, 1, \frac{c_B}{e_B} \right), \\
& F^\textnormal{index} \left(-\frac{c_A^2}{e_B^2} + \frac{2\, c_B c_A}{e_B e_A} - \frac{e_B}{e_A}, 1, -\frac{c_A e_A}{e_B^2} + \frac{c_B}{e_B} \right), \\
& F^\textnormal{index} \left( -\frac{e_A^2 c_A}{e_B^3} - \frac{2 \, c_A^2 - c_B e_A }{e_B^2} + \frac{2 \, c_B c_A}{e_B e_A} - \frac{e_B}{e_A}, 1, -\frac{e_A^3}{e_B^3} - \frac{2 \, c_A e_A}{e_B^2} +\frac{c_B}{e_B}  \right)  \bigg\} \\
\leq  \ & F^\textnormal{index} \left(-\frac{e_B}{e_A},1,0\right);
\end{align*}
\begin{align*}
\sigma_{12}  =   \min \bigg\{ & F^\textnormal{index} \left(1,-\frac{e_B}{e_A},0\right),  \\
& F^\textnormal{index}  \left( -\frac{e_A}{e_B}, \frac{e_A^2}{e_B^2} + \frac{c_A}{e_B}, 0 \right), \\
& F^\textnormal{index} \left(\frac{e_A^2}{e_B^2} + \frac{c_A}{e_B}, -\frac{e_A^3}{e_B^3} - \frac{2 \, c_A e_A}{e_B^2} + \frac{c_B}{e_B}, 1 \right),
F^\textnormal{index} \left( -\frac{c_B e_A}{e_B^2}, \frac{c_B e_A^2}{e_B^3} + \frac{c_B c_A}{e_B^2}, 0  \right), \\
& F^\textnormal{index} \left(  \frac{e_A^2}{e_B^2} + \frac{2 \, c_A}{e_B} -\frac{c_B}{e_A}, -\frac{e_A^3}{e_B^3} - \frac{3 \, c_A e_A}{e_B^2} - \frac{c_A^2}{e_B e_A} + \frac{2 \, c_B}{e_B}, 1 \right)  \bigg\} \\
\leq  \ & F^\textnormal{index} \left(1,-\frac{e_B}{e_A},0\right); 
\end{align*}
\begin{align*}
\sigma_{25}  =   \min \bigg\{ & F^\textnormal{index} \left(1,-\frac{e_B}{e_A},0\right),  \\
& F^\textnormal{index}  \left(-\frac{e_A}{e_B}, \frac{e_A^2}{e_B^2} +  \frac{c_A}{e_B}, 1  \right), F^\textnormal{index} \left( \frac{c_B}{e_B}, -\frac{c_B e_A}{e_B^2} , 0 \right), \\
& F^\textnormal{index} \left( -\frac{c_A}{e_A} - \frac{e_A}{e_B}, \frac{e_A^2}{e_B^2} + \frac{2 \, c_A}{e_B} - \frac{c_B}{e_A}, 1 \right), \\
& F^\textnormal{index} \left( \frac{c_A^2 c_B}{e_B^2 e_A} + \frac{c_B^2}{e_B^2} - \frac{c_A}{e_A} - \frac{e_A}{e_B}, - \frac{c_A^2 c_B + c_B^2 e_A}{e_B^3} + \frac{c_B^2 c_A}{e_B^2 e_A} + \frac{e_A^2}{e_B^2} + \frac{2 \, c_A}{e_B} -\frac{c_B}{e_A}, 1 \right)  \bigg\}\\
\leq  \ & F^\textnormal{index} \left(1,-\frac{e_B}{e_A},0\right);
\end{align*}
and
\begin{align*}
\sigma_{53}  =   \min \bigg\{ & F^\textnormal{index} \left(1,-\frac{e_B}{e_A}, 0\right),  \\
& F^\textnormal{index}  \left( 1 , -\dfrac{c_A}{e_A} - \dfrac{e_A}{e_B} , -\dfrac{e_B}{e_A} \right), \\
& F^\textnormal{index} \left( 1 , \dfrac{c_A^2 c_B}{e_B^2 e_A} + \dfrac{c_B^2}{e_B^2} - \dfrac{c_A}{e_A} - \dfrac{e_A}{e_B} , \dfrac{c_B c_A}{e_B e_A} - \dfrac{e_B}{e_A}  \right), \\
& F^\textnormal{index} \left(  1, -\dfrac{c_A^3 + e_A c_B c_A}{e_B^3} + \dfrac{2 \, c_A^2 c_B}{e_B^2 e_A} + \dfrac{c_B^2}{e_B^2} - \dfrac{c_A}{e_A} - \dfrac{e_A}{e_B} , - \dfrac{c_A^2}{e_B^2} + \dfrac{2 \, c_B c_A}{e_B e_A} - \dfrac{e_B}{e_A}  \right) \bigg\} \\
\leq \ & F^\textnormal{index} \left(1,-\frac{e_B}{e_A},0\right).
\end{align*}

\subsection{Proposition~\ref{prop:stab_as}}\label{subsec:stab_cycles_as_network}


The relations (a2) and (a3) in Lemma~5.6 determine the first and third conditions in Proposition~5.2. The remaining condition is $c_A e_A - c_B e_B < 0$. If it is satisfied, then $\sigma_\text{R-to-P}=-\infty$, otherwise $\sigma_\text{R-to-P}>-\infty$. In the latter case, we further check the sign of~$\sigma_\text{R-to-P}$. Taking~(20) negative entries can only occur in the last row of each $(\M_2)^j$, $j=1,\ldots,5$. It follows that 
\begin{align*}
\sigma_\text{R-to-P}   =  \\
  \min \bigg\{ 
& F^\textnormal{index} \left(-\frac{e_B}{e_A},1,0\right),  \\
& F^\textnormal{index}  \left(-\frac{e_B c_B}{e_A^2} + \frac{c_A}{e_A}, 0, -\frac{e_B}{e_A} \right), \\
& F^\textnormal{index} \left(-\frac{c_B^2 e_B}{e_A^3} + \frac{c_A c_B + e_B^2}{e_A^2}, -\frac{e_B}{e_A}, -\frac{e_B c_B}{e_A^2} + \frac{c_A}{e_A} \right), \\
& F^\textnormal{index} \left( -\frac{c_B^3 e_B}{e_A^4} + \frac{c_B^2 c_A + 2 \, e_B^2 c_B}{e_A^3}  - \frac{2 \, c_A e_B}{e_A^2}, -\frac{e_B c_B}{e_A^2} + \frac{c_A}{e_A}, -\frac{c_B^2 e_B}{e_A^3} + \frac{c_A c_B + e_B^2}{e_A^2} \right), \\
& F^\textnormal{index} \, \bigg( \begin{aligned}[t] & -\frac{c_B^4 e_B}{e_A^5} + \frac{c_B^3 c_A + 3 \, c_B^2 e_B^2}{e_A^4} - \frac{4\, c_A c_B e_B + e_B^3}{e_A^3} +\frac{c_A^2}{e_A^2}, -\frac{c_B^2 c_A}{e_A^3} + \frac{c_A c_B + e_B^2}{e_A^2}, \\
& -\frac{c_B^3 e_B}{e_A^4} + \frac{c_B^2 c_A + 2\, e_B^2 c_B}{e_A^3} - \frac{2\, c_A e_B}{e_A^2} \bigg) \bigg\}. \end{aligned} 
\end{align*}

According to the function $F^\text{index}$, 
consider the rows in $(\M_2)^j$ with at least with negative entries and define the sums of the row elements as follows:
\begin{align*}
\mathfrak{s}_1 & = -\frac{e_B}{e_A} + 1 \\[0.1cm]
\mathfrak{s}_2 & =  -\frac{e_B c_B}{e_A^2} + \frac{c_A}{e_A} -\frac{e_B}{e_A} \\[0.1cm]
\mathfrak{s}_3 & = -\frac{c_B^2 e_B}{e_A^3} + \frac{c_A c_B + e_B^2}{e_A^2} -\frac{e_B}{e_A} -\frac{e_B c_B}{e_A^2} + \frac{c_A}{e_A} \\[0.1cm]
\mathfrak{s}_4 & =  -\frac{c_B^3 e_B}{e_A^4} + \frac{c_B^2 c_A + 2 \, e_B^2 c_B}{e_A^3}  - \frac{2 \, c_A e_B}{e_A^2} -\frac{e_B c_B}{e_A^2} + \frac{c_A}{e_A} -\frac{c_B^2 e_B}{e_A^3} + \frac{c_A c_B + e_B^2}{e_A^2} \\[0.1cm]
\mathfrak{s}_5 & = \begin{aligned}[t] & -\frac{c_B^4 e_B}{e_A^5} + \frac{c_B^3 c_A + 3 \, c_B^2 e_B^2}{e_A^4} - \frac{4\, c_A c_B e_B + e_B^3}{e_A^3} +\frac{c_A^2}{e_A^2} -\frac{c_B^2 c_A}{e_A^3} + \frac{c_A c_B + e_B^2}{e_A^2} \\
&-\frac{c_B^3 e_B}{e_A^4} + \frac{c_B^2 c_A + 2\, e_B^2 c_B}{e_A^3} - \frac{2\, c_A e_B}{e_A^2}.\end{aligned}
\end{align*}
A trivial verification shows that
\[
\mathfrak{s}_2 \leq \mathfrak{s}_3 < \frac{c_A e_A - c_B e_B}{e_A^2} < \mathfrak{s}_5 \leq \mathfrak{s}_4.
\]
By virtue of $e_B < e_A$ and $c_A e_A - c_B e_B>0$, we get immediately $\mathfrak{s}_1>0$ and $\mathfrak{s}_4 \geq \mathfrak{s}_5 > 0$. Hence,
\begin{align*}
& F^\textnormal{index} \left(-\frac{e_B}{e_A},1,0\right)   > 0 \\
& F^\textnormal{index} \left( -\frac{c_B^3 e_B}{e_A^4} + \frac{c_B^2 c_A + 2 \, e_B^2 c_B}{e_A^3}  - \frac{2 \, c_A e_B}{e_A^2}, -\frac{e_B c_B}{e_A^2} + \frac{c_A}{e_A}, -\frac{c_B^2 e_B}{e_A^3} + \frac{c_A c_B + e_B^2}{e_A^2} \right)   > 0 \\
&F^\textnormal{index} \, \bigg( \begin{aligned}[t] & -\frac{c_B^4 e_B}{e_A^5} + \frac{c_B^3 c_A + 3 \, c_B^2 e_B^2}{e_A^4} - \frac{4\, c_A c_B e_B + e_B^3}{e_A^3} +\frac{c_A^2}{e_A^2}, -\frac{c_B^2 c_A}{e_A^3} + \frac{c_A c_B + e_B^2}{e_A^2}, \\
& -\frac{c_B^3 e_B}{e_A^4} + \frac{c_B^2 c_A + 2\, e_B^2 c_B}{e_A^3} - \frac{2\, c_A e_B}{e_A^2} \bigg) > 0. \end{aligned}
\end{align*}
We are now reduced to three possibilities: 
\begin{enumerate}
\item[(I)] $0< \mathfrak{s}_2 \leq \mathfrak{s}_3$, which is equivalent to $c_A e_A - c_B e_B> e_A e_B$. We obtain
\begin{align*}
& F^\textnormal{index}  \left(-\frac{e_B c_B}{e_A^2} + \frac{c_A}{e_A}, 0, -\frac{e_B}{e_A} \right) > 0\\
& F^\textnormal{index} \left(-\frac{c_B^2 e_B}{e_A^3} + \frac{c_A c_B + e_B^2}{e_A^2}, -\frac{e_B}{e_A}, -\frac{e_B c_B}{e_A^2} + \frac{c_A}{e_A} \right) > 0
\end{align*}
and consequently $\sigma_\text{R-to-P} > 0$.
\item[(II)] $\mathfrak{s}_2 \leq 0 \leq \mathfrak{s}_3$. We obtain
\begin{align*}
& F^\textnormal{index}  \left(-\frac{e_B c_B}{e_A^2} + \frac{c_A}{e_A}, 0, -\frac{e_B}{e_A} \right) < 0\\
& F^\textnormal{index} \left(-\frac{c_B^2 e_B}{e_A^3} + \frac{c_A c_B + e_B^2}{e_A^2}, -\frac{e_B}{e_A}, -\frac{e_B c_B}{e_A^2} + \frac{c_A}{e_A} \right) > 0
\end{align*}
and consequently $\sigma_\text{R-to-P} < 0$.
\item[(III)] $\mathfrak{s}_2 \leq \mathfrak{s}_3 < 0$. We obtain
\begin{align*}
& F^\textnormal{index}  \left(-\frac{e_B c_B}{e_A^2} + \frac{c_A}{e_A}, 0, -\frac{e_B}{e_A} \right) < 0\\
& F^\textnormal{index} \left(-\frac{c_B^2 e_B}{e_A^3} + \frac{c_A c_B + e_B^2}{e_A^2}, -\frac{e_B}{e_A}, -\frac{e_B c_B}{e_A^2} + \frac{c_A}{e_A} \right) < 0
\end{align*}
and consequently $\sigma_\text{R-to-P} < 0$.
\end{enumerate}

In Proposition~\ref{prop:ind-Star} only the first condition in (a) and (b) is determined by~\eqref{hyp:as}. When the stability index $\sigma_\text{Star}$ is finite, it is negative since $e_B<e_A$ and hence $F^\text{index} \left(1,-\frac{e_A}{e_B},0\right) = 1 - \frac{e_A}{e_B} <0$. The Star cycle is at most f.a.s.

From Proposition~\ref{prop:index-RSP},
the stability indices for $\Sigma_\text{RSP}$ reduce to
\begin{align*}
\sigma_{31} &  = \begin{cases}
1 - \dfrac{e_Bc_B}{e_A^2} + \dfrac{c_A}{e_A} \; \;  (<0) & \text{ if }  - \dfrac{e_Bc_B}{e_A^2} + \dfrac{c_A}{e_A}  \leq - 1 \\[0.5cm]
\min \left\{ -1 + \dfrac{e_A}{e_B},  -1 + \dfrac{e_A^2}{e_B c_B - c_A e_A }\right\}\; \;  (>0) & \text{ if }  -1 < - \dfrac{e_Bc_B}{e_A^2} + \dfrac{c_A}{e_A}  < 0 
\end{cases} \\
\sigma_{12}  & = -1 + \dfrac{e_A}{e_B}   \; \;  (>0) \\ 
\sigma_{23} & = 1- \frac{c_A}{e_A}  -  \frac{e_A}{e_B}   \; \;  (<0).
\end{align*}
It is easy to see that at least $\sigma_{23}<0$ always preventing this sub-cycle from being e.a.s.

\end{document}